\documentclass[12pt,a4paper]{article}

\usepackage[utf8]{inputenc}
\usepackage[T2A]{fontenc}
\usepackage[russian]{babel}
\usepackage{hyphenat}
\usepackage{indentfirst}
\usepackage{wrapfig}
\usepackage{caption}
\usepackage{subcaption}
\usepackage{graphicx}
\usepackage{amsfonts}
\usepackage{amsmath}
\usepackage{amssymb}
\usepackage{amsthm}
\usepackage{hyperref}
\hypersetup
{
colorlinks=true,
linkcolor=black,
filecolor=black,
urlcolor=black,
citecolor=black,
}

\newcommand{\R}{\mathbb{R}}
\newcommand{\Z}{\mathbb{Z}}
\newcommand{\md}{\frac{k-1}{2}} %middle
\newcommand{\mdd}{(k-1)/2} %middle
\newcommand{\prd}[1]{\left(#1\right)_{\sigma}} %permutation product
\newcommand{\prdd}[1]{\left(#1\right)_{\sigma^{'}}}
\newcommand{\Gr}{G_{rot}}
\newcommand{\oG}{\overline{G}_{rot}}
\newcommand{\class}{(m,j_1,\dots,j_{n-m})}
\newcommand{\spec}{(1,j_1,\dots,j_{n-1},q)}

\newtheorem{deff}{Определение}
\newtheorem{theorem}{Теорема}
\newtheorem{lemma}{Лемма}
\newtheorem*{remark}{Замечание}
\newtheorem*{cons}{Следствие}
\newtheorem*{example}{Пример}

\newtheorem*{permutationlemma}{Лемма о положении}
\newtheorem*{Class}{Теорема классификации}
\newtheorem*{twogroups}{Теорема о двух группах}
\newtheorem*{orientationlemma}{Лемма об ориентации}
\newtheorem*{cluster}{Теорема о кластерах}
\newtheorem*{carefulcomb}{Принцип аккуратной перестановки}
\newtheorem*{Classpermutation}{Теорема о чётной перестановке класса}
\newtheorem*{Classorient}{Теорема об ориентировании класса}

\begin{document}

\title{Полная система инвариантов многомерного Кубика Рубика}

\author{Исаев Роман}

\date{26 июня 2021 г.}

\maketitle

\begin{abstract}
При изучении Кубика Рубика, как правило, задаются вопросами об оптимальном алгоритме сборки, но в то же время важнейшие групповые особенности головоломки упускаются из виду.

Кубик Рубика обладает набором инвариантов. Это те величины, которые сохраняются при вращении слоёв, но могут измениться, если некоторые кубики извлечь из головоломки и поменять местами (или сменить их ориентацию). В этой работе мы обобщим головоломку на случай произвольной размерности и произвольной длины ребра, после чего опишем все инварианты. Именно их существование объясняет, почему не от каждого состояния головоломки можно прийти к другому поворотами слоёв.
\end{abstract}

\section{Введение}
Кубик Рубика изобрёл Эрне Рубик в 1974 году для наглядного представления работы теории групп. Несмотря на почти пятидесятилетнюю историю существования, он продолжает притягивать к себе как любителей головоломок, так и исследователей. Со временем появились различные его вариации: профессорский кубик, пирамидка, мегаминкс и сотни других, многие из которых получили достаточно широкое распространение. Относительно недавно любителей головоломок заинтересовал многомерный Кубик Рубика. Он значительно отличается от классического. Например, в нём появляются новые типы кубиков и эффекты, возникновение которых объясняет теория групп и комбинаторика.

Для демонстрации особенностей головоломки рассмотрим следующую ситуацию. В стандартном Кубике Рубике ($3 \times 3 \times 3$) извлечём угловой кубик, развернём на 120 градусов и вернём на место. Хорошо известно, что из нового состояния поворотами слоёв достичь старого не выйдет. Этот пример иллюстрирует то, что угловые кубики головоломки обладают инвариантом, связанным с их ориентацией. В 'многомерье' ситуация иная: данный эффект угловых элементов проявляет себя в четырёхмерном Кубике, но для всех размерностей выше отсутствует. Как покажем позже, не только угловые кубики имеют пропадающие инварианты.

Как правило, изучение Кубика Рубика обычно сводится к решению следующих двух проблем:

\begin{itemize}
  \item Описание полной системы инвариантов
  \item Оценка (или точное значение) минимального числа ходов для сборки Кубика
\end{itemize}

В обеих проблемах имеются крупные продвижения. Полная система инвариантов трёхмерного Кубика Рубика описана в статье \cite{Bonzio}, там же приведены достаточные и необходимые условия для осуществления его правильной сборки. В решении многомерной проблемы успехи менее значительные, в статье \cite{Levashev} описаны полные системы инвариантов многомерных Кубиков Рубика с ребрами длины 2 и 3.

Наименьшее количество ходов для сборки было найдено только для стандартного Кубика Рубика при помощи ЭВМ. В общем случае (для трёхмерной проблемы) в статье \cite{Erik} была найдена замечательная оценка: $\Theta(k^2 / \log \: k)$, где $k$ -- длина ребра. Для многомерного случая подобные исследования не проводились.

В данной работе описано решение первой проблемы в общем случае, то есть \textbf{полная система инвариантов} произвольного Кубика Рубика. Комбинированием идей статьи \cite{Erik} и результатов, которые мы получим позже, вероятно, можно получить оценку на число ходов и в многомерном случае. Под \textbf{полной системой инвариантов Кубика} понимается набор всех инвариантов с указанием множества значений, которое принимает каждый из них.

Прежде всего, начнём с определения наиболее важных объектов и операций, после чего приступим к терминологии и формулировке результатов.

\section{Основные обозначения и определения}

Далее всюду под натуральными $k \geq 2$ и $n \geq 3$ всегда будем подразумевать \textbf{длину ребра} Кубика Рубика и его \textbf{размерность} соответственно.

\subsection{Используемые множества}

\begin{itemize}
  \item \textbf{Индексные}: $M=\{0,1,\dots,k-1\}, F=M \setminus \{0,k-1\},$
  
  $L=\{1,\dots,\left[\md \right]\}$, где квадратными скобками обозначена целая часть числа. 
  \item \textbf{Основные}: $E=[0,k]$ и $E_j=[j,j+1], \; \forall j\in M$, используемые для обозначения основных элементов головоломки, их положений и ориентаций.
\end{itemize}

Ниже представлен список используемых семейств множеств. Для простоты записи элементы одного и того же семейства будем обозначать одинаковыми символами. То есть, $E_M$ и $E_M$ не являются одним и тем же множеством в общем случае. В редких случаях, когда будет важно подчеркнуть схожесть или различие элементов из одного семейства, это будет оговорено.

\begin{itemize}
  \item $V \in \{ \{0\}, \{k\} \}$,
  \item $E_M \in \Omega_M = \{E_i \; | \; i\in M\}$,
  \item $E_V \in \Omega_{V} = \{E_0, E_{k-1}\}$,
  \item $E_F \in \Omega_F =\Omega_M\setminus \Omega_{V}$
  \item $P_j \in \Omega_j = \{E_j, E_{k-1-j} \}$ для любого $j \in L$.
  
  $P_\md=E_\md$ и $\Omega_\md=\{E_\md\}$, когда $j=\mdd$.
\end{itemize}

В этой работе часто будет появляться величина $\mdd$. Договоримся, что если она фигурирует в наших утверждениях, то подразумевается дополнительное условие, что $k$ нечётно. Если же говорится о целой части $\mdd$, то $k$, конечно, может быть произвольным.

Также, когда речь пойдёт о количестве каких-либо элементов, будем придерживаться соглашений по \textbf{биномиальному коэффициенту}: $\binom{n}{s}$ обращается в нуль при $n\leq 0$, или $s<0$, или $s>n$.

\subsection{Операции над множествами}
\begin{deff}
Образ множества $A \subset \R$ при отображении $f(x)=k-x$ будем называть \textbf{противоположным} к $A$ и обозначать за $\overline{A}$.
\end{deff}

\begin{example}
$E_i$ противоположно $E_{k-1-i}$ при любых $i \in M$. Отсюда следует, что после применения такой операции из $E_V$ и $P_j$ мы получим снова $E_V$ и $P_j$ соответственно.
\end{example}

\begin{deff}
Для множеств $A_1,\dots,A_n \subset \R$ и произвольной перестановки $\sigma$ определено \textbf{перестановочное произведение}
\[
\prd{A_1 A_2 \dots A_n}=A_{\sigma(1)}\times A_{\sigma(2)} \times \dots \times A_{\sigma(n)}.
\]
\end{deff}

\begin{example}
\[
\prd{A_1 A_2}=\begin{cases} A_1\times A_2, & \sigma = e, \\ A_2\times A_1, & \sigma = (1 \; 2). \end{cases}
\]
\end{example}

\begin{remark}
Будем придерживаться соглашения, что $E_M^n$ - это не перемножение одного и того же множества, а
    \[E_M^n = E_M \times \dots \times E_M = E_{i_1} \times \dots \times E_{i_n}\]
    для некоторых $i_j$ из $M$.
Аналогичное верно и для других множеств, таких как $V, \: E_V, \: E_F$ и $P_j$ .
\end{remark}

\begin{deff}
Пусть $A_1,\dots,A_n \subset \R$, а $\Omega_1,\dots,\Omega_m$ - семейства подмножеств $\R$. Тогда их \textbf{смешанным произведением} называется
\[
A_1\times \dots \times A_n\times \Omega_1\times \dots \times \Omega_m = \{A_1\times \dots \times A_n\times B_1\times \dots \times B_m \; | \; B_i\in \Omega_i\}.
\]
\textbf{Смешанное перестановочное произведение} определяется аналогично перестановочному произведению. 

\textbf{Смешанное перестановочное семейство}
\[
B_1 \dots B_n \Omega_1 \dots \Omega_m = \bigcup_{\sigma \: \in \: S_{n+m}} \prd{B_1 \dots B_n \Omega_1 \dots \Omega_m}.
\]
\end{deff}

\begin{remark}
В частности, $n$ может быть равным нулю.
\end{remark}

\begin{example}
\[
\prd{E_M\Omega_M}=
\begin{cases} E_M\times \Omega_M=\{E_M\times E_i \; | \; i \in M \}, & \sigma = e, \\ \Omega_M\times E_M= \{E_i\times E_M \; | \; i \in M \}, & \sigma = (1 \; 2). \end{cases}
\]
\[
E_M \Omega_M = E_M \times \Omega_M \cup \Omega_M \times E_M.
\]
\[
\Omega_M^2 = \Omega_M \times \Omega_M = \{E_i\times E_j \; | \; i,j \in M \}.
\]
\end{example}

\section{Анатомия многомерного Кубика Рубика}

\subsection{Терминология}

\begin{deff}
\textbf{$n$-мерным кубом с длиной ребра $k$} (далее -- кратко, \textbf{кубом}) будем называть подмножество $\R^n$
\[
E^n = \{(x_1, \dots ,x_n) \; | \; 0 \le x_i \le k\}.
\]

Обобщим понятие грани куба на многомерный случай.

Множество граничных точек куба, образующих $t$-мерное подпространство в $\R^n$ $(t \geq 2)$, называется \textbf{$t$-мерной гранью} (или \textbf{$t$-гранью}) куба. Грань записывается как
\[
\prd{V^{n-t} E^t} = \{(x_1, \dots ,x_n) \; | \; x_{i_1}=c_1, \dots ,x_{i_{n-t}}=c_{n-t}\},
\]
где значения $(n-t)$ выбранных координат фиксированы и $c_1, \dots ,c_{n-t} \in \{0,k\}$.

В частности, если положить $t=0$ или $t=1$, то получим \textbf{вершину} или \textbf{ребро} соответственно.
\end{deff}

\begin{deff}
Разделим куб на $k^n$ единичных \textbf{кубиков}, каждый из которых является подвижным объектом в головомке и может менять своё положение. В свою очередь \textbf{положением кубика}, координаты точек которого 
\[
j_i \le x_i \le j_i+1, \; i=1, \dots, n,
\]
называется выражение
\[
E_M^n = E_{j_1}\times \dots \times E_{j_n}
\]
для некоторой совокупности индексов $j_1,\dots,j_n \in M$.

Определим \textbf{грани кубика} точно так же, как и для куба. \textbf{Внешней $t$-гранью} будем называть лишь ту $t$-мерную грань кубика, все точки которой -- суть граничные точки большого куба. Иначе говоря, внешняя грань содержится в некоторой $t$-грани куба и получается её пересечением с кубиком.

\end{deff}

\begin{remark}
Так как в головоломке покрашены лишь внешние грани, то и работать мы будем исключительно с ними. Потому с этого момента, для простоты, будем называть их гранями. 
\end{remark}

Пример того, как обозначаются элементы головоломки в новой терминологии можно увидеть на рисунке \ref{sec3sub1pic1}.

\begin{figure}[h]
\begin{center}
\begin{minipage}[h]{0.6\linewidth}
\includegraphics[width=1\linewidth]{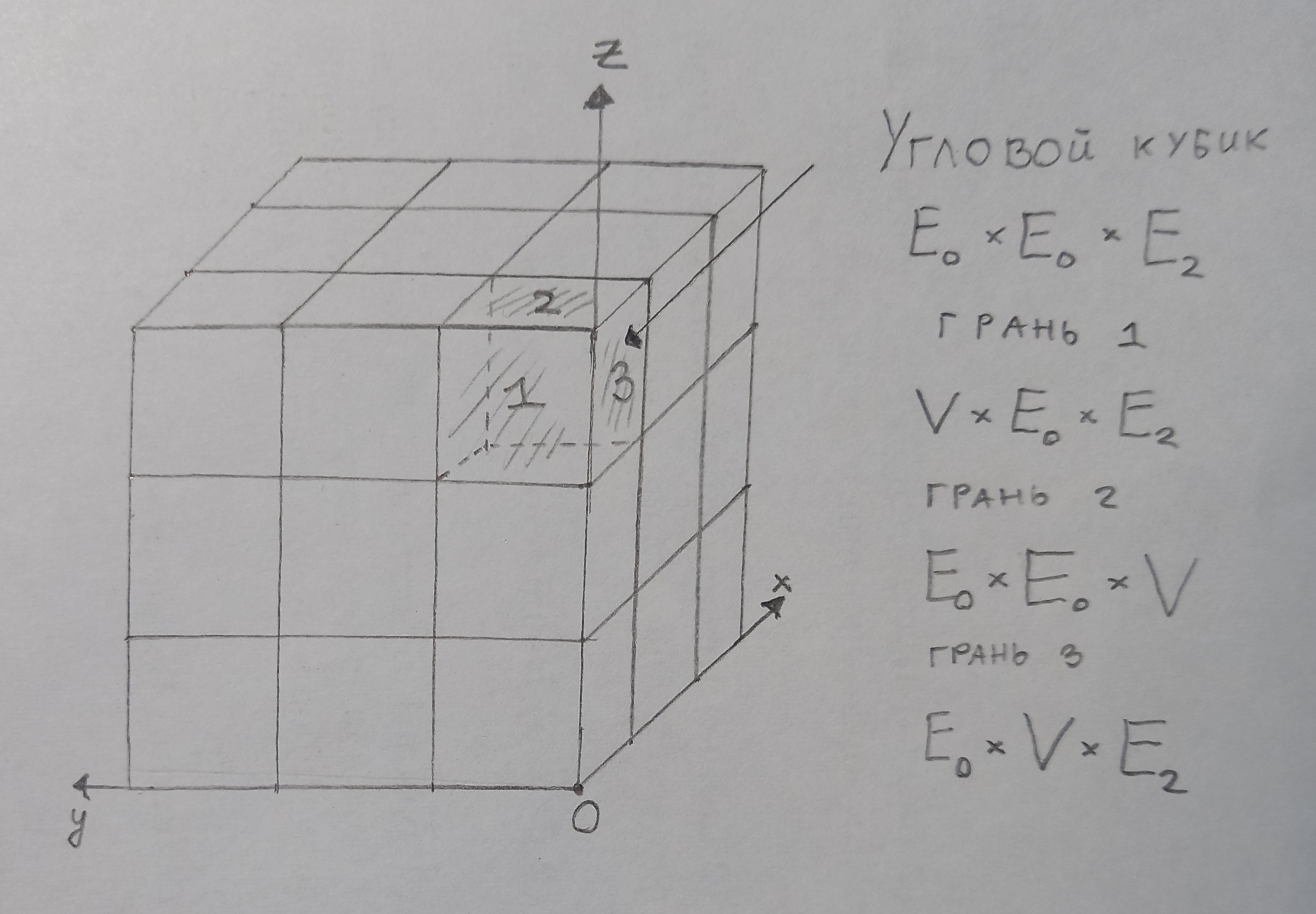}
\caption{}
\label{sec3sub1pic1}
\end{minipage}
\begin{minipage}[h]{0.39\linewidth}
\includegraphics[width=1\linewidth]{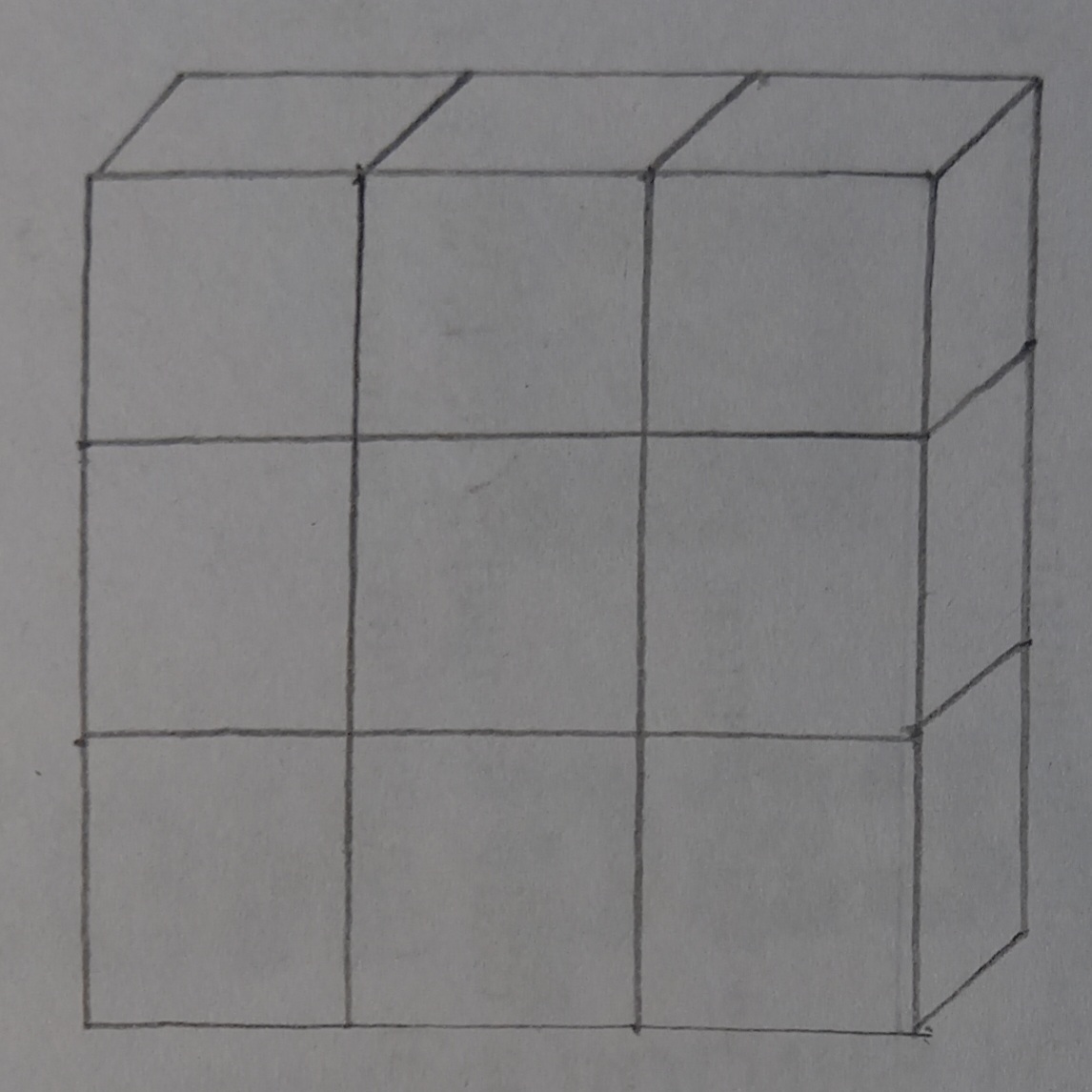}
\caption{}
\label{sec3sub1pic2}
\end{minipage}
\end{center}
\end{figure}

\begin{deff}
Рассмотрим множество всех кубиков, у центров которых зафиксированы и совпадают значения $(n-t)$ координат, а остальные $t$ могут принимать произвольные значения. Полученное множество, являющееся аналогом $t$-мерного куба внутри $n$-мерной головоломки, называется \textbf{$t$-мерным слоем} и записывается как
\[
\prd{E_M^{n-t} \Omega_M^t}.
\]

\textbf{Внешним $t$-мерным слоем} называется слой, состоящий из всех кубиков, примыкающих к некоторой $t$-грани куба. То есть, если задана грань куба
\[
\prd{V^{n-t} E^t},
\]
то слой можно записать в виде
\[
\prd{E_V^{n-t} \Omega_M^t}.
\]
Другие же слои называются \textbf{внутренними}.
\end{deff}

На рисунке \ref{sec3sub1pic2} продемонстрирована проекция двухмерного слоя многомерного куба на трёхмерное пространство. Проекция трёхмерного слоя является обычным кубом, который можно увидеть на рисунке \ref{sec3sub1pic1}.

\begin{deff}
\textbf{Оператором поворота пространства} называется отображение $\Psi_{ij}: \R^n \to \R^n$
\[
\Psi_{i,j}\left(\sum\limits_{k=1}^{n} x_k e_k\right) =\sum\limits_{k=1, \: k\neq i,j}^{n} (x_k e_k) - x_j e_i + x_i e_ j,
\]
где $e_1, \dots ,e_n$ -- стандартный базис $\R^n$.

Выберем некоторый слой и ограничим действие данного оператора только на него. Тогда все кубики вне слоя останутся неподвижными. Новый оператор называется \textbf{оператором поворота данного $t$-слоя}
\[
\psi_{i,j}:  \Omega_M^n \to \Omega_M^n
\]
и изменяет положения кубиков слоя по следующему правилу
\[
\psi_{i,j}\left(\prod_{i=1}^{n} A_m\right) = \prod_{i=1}^{n} A^{'}_m
\]
\[
A^{'}_m=\begin{cases}
A_m, & m\neq i,j, \\
\overline{A}_j, & m=i, \\
A_i, & m=j.
\end{cases}
\]
\end{deff}

Далее под поворотом будем всегда подразумевать оператор поворота слоя.

Заметим, что то, какой слой, содержащий кубик, будет выбран, не влияет на результат поворота. Новое положение зависит только от старого и параметров $i$ и $j$. Потому, если не оговорена размерность вращающегося слоя, будем подразумевать двумерный слой. Этот выбор обусловлен тем, что двумерные слои изменяют положения наименьшего количества кубиков, что обеспечивает большую 'аккуратность' в обращении с элементами головоломки.

\begin{deff}
\textbf{Комбинацией поворотов} (или просто, комбинацией) будем называть последовательность поворотов слоёв головоломки.

Каждая комбинация $q$ имеет \textbf{обратную}, записываемую как $q^{-1}$. Она осуществляется поворотами слоёв в обратном порядке.
\end{deff}

\begin{deff}
Для каждого $2 \leq t \leq (n-1)$ покрасим $t$-грани всех кубиков в разные цвета так, чтобы
\begin{itemize}
    \item каждый цвет присутствовал только на гранях определённой размерности,
    \item среди граней одинаковой размерности каждый цвет встречался одно и то же число раз.
\end{itemize}
Множество всех покрашенных граней с оператором поворота слоя и будем называть \textbf{Кубиком Рубика}.
\end{deff}

\begin{deff}
\textbf{Состоянием Кубика Рубика} называется расположение в головоломке покрашенных граней.

Два \textbf{состояния} называются \textbf{связанными}, если от одного можно прийти к другому некоторой комбинацией.

Два кубика называются \textbf{связанными}, если поворотами можно перевести кубик из своего положения в положение другого. Такое отношение, очевидно, является отношением эквивалентности.
\end{deff}

Отметим следующее. Пусть в Кубике Рубика найдутся одинаково покрашенные кубики. Пронумеруем все кубики и перейдём поворотами от одного состояния к другому, в котором головоломка раскрашена также, но кубики располагаются иначе. Тогда согласно определению начальное и конечное состояние совпадают, хотя сами кубики уже расположены иначе.

\begin{deff}
\textbf{Инвариантом Кубика Рубика} называется величина, зависящая от его состояния. Она сохраняется при поворотах слоёв, но существуют такие состояния куба, в которых она принимает разные значения.

\textbf{Система инвариантов} -- термин, означающий то же, что множество инвариантов. \textbf{Полная система} - система с максимальным числом инвариантов.
\end{deff}

\subsection{Классификация по положениям}

Предположим дано положение кубика
\[
A_1 \times \dots \times A_n,
\]
где $A_i=E_l, \: A_j=E_r$. Изучим действие поворота $\psi_{i,j}$ на положение. Так как будут изменяться лишь $i$-ая и $j$-ая координаты, то легче рассматривать положение кубика в упрощённой форме: $A_i \times A_j$ = $E_l \times E_r$.

\begin{permutationlemma}
Поворот не меняет положения кубика при $l=r=\mdd$. В противном случае осуществляется циклическая перестановка четырёх кубиков, включая данный, положения которых
\[
\{E_l\times E_r,\; \overline{E_r}\times E_l,\; \overline{E_l}\times \overline{E_r},\; E_r\times \overline{E_l} \}.
\]
Иначе говоря, поворотами $\psi_{i,j}$ можно привести кубик к одному из этих четырёх положений.
\end{permutationlemma}

Доказательство мгновенно вытекает из определения поворота.

Введение выше упрощённой ориентации будет использоваться и в дальнейшем для упрощения записи положений кубиков.

\begin{deff}
Для произвольного кубика, находящегося в положении
\[
\prd{E_V^m P_{j_1} \dots P_{j_{n-m}}},
\]
определим \textbf{каноническое положение}
\[
P=E_0^m \times E_{k_1}\times \dots \times E_{k_{n-m}},
\]
где индексы $k_1,\dots,k_{n-m}$ -- это те же самые индексы $j_1,\dots,j_{n-m} \in L$, но записанные в порядке неубывания.
\end{deff}

Легко проверить, что каноническое положение инвариантно относительно поворота. Всюду в этой работе будем обозначать его как $P$. Далеко не всегда от произвольного положения кубика можно комбинацией прийти к каноническому, что будет следствием следующей теоремы.

Для начала, вспомним, что отношение связанности является отношением эквивалентности. В первом пункте теоремы разобьём множество всех кубиков на подмножества, которые назовём тривиальными классами. В общем случае они классами эквивалентности по связанности не являются. Последние будут описаны во втором пункте теоремы.

\begin{Class}

\textbf{1. Тривиальная классификация}.

Для любого целого $m$ такого, что $0\leq m\leq n$, определим тривиальный класс $(m)$, состоящий из всех кубиков с положениями
\[
\prd{E_V^m E_F^{n-m}}.
\]
Тогда множество всех кубиков распадается на непересекащиеся тривиальные классы.

Кубики класса (0) называем \textbf{внутренними}, остальные -- \textbf{внешними}.

\textbf{2. Классификация по связанности}.

Все классы эквивалентности по связанности, на которые распадается множество внешних кубиков, описаны ниже.

Для любого $1\leq m\leq n$ и неубывающего набора чисел
\[
j_1 \leq \dots \leq j_{n-m} \in L,
\]
называемых \textbf{характеристиками}, рассмотрим множество $\class$ всех кубиков с положениями
\[
\prd{E_V^m P_{j_1}\dots P_{j_{n-m}}}.
\]
Если $m=n$, то набор характеристик вырождается и множеством будет тривиальный класс $(n)$.

\textbf{Если выполняется условие}
\[
m=1 \text{ и } j_1 < j_2 <\dots < j_{n-1}<\md,
\]
($\mdd$ здесь не обязательно натуральное) то такое множество разбивается на два класса эквивалентности, в зависимости от того, приводятся ли его кубики в свои канонические положения или нет. Оба класса называются \textbf{особыми} и существуют только при $k\ge 2n$.

\textbf{Если оно не выполняется,} то данное множество само является классом.

\end{Class}

\begin{proof}
\textbf{Тривиальная классификация.}

Легко проверить, что вращение слоёв не влияет на количество множеств $E_V$ в положение произвольно взятого кубика. Поэтому, в зависимости от количества $E_V$ в записи положения, кубик принадлежит одному из $(n+1)$ тривиальных классов.

\textbf{Классификация по связанности.}

\textbf{1.} Возьмём элемент тривиального класса $(m)$. Так как для некоторого $j \in F$
\[
E_F = E_j = P_s, \: s=min(j,k-1-j) \in L,
\]
то положение можно записать в виде
\[
\prdd{E_V^m E_F^{n-m}}=\prd{E_V^m P_{j_1}\dots P_{j_{n-m}}}
\]
для некоторого $\sigma$ и $j_1 \leq \dots \leq j_{n-m}$.

Итак, \textbf{каждый кубик принадлежит хотя бы одному из описанных множеств.}

\textbf{2.} Для любого кубика из $\class$ каноническое положение будет
\[
P=E_0^m \times E_{j_1} \times \dots \times E_{j_{n-m}}.
\]
То есть оно зависит не от выбора кубика из класса, а только от характеристик. Так как для разных множеств характеристики разные, то и их канонические положения различны. В виду инвариантности последних, получаем, что \textbf{множества не пересекаются}.

\textbf{3.} Рассмотрим произвольный кубик из $\class$ и попробуем привести его к каноническому положению класса. Если это удастся, то все элементы множества будут связаны, а само множество -- классом эквивалентности.

\textbf{Шаг 1.} Так как противоположными к $E_V$ и $P_j$ будут $E_V$ и $P_j$ соответственно, то поворот $\psi_{i,j}$ совершает транспозицию $i$-го и $j$-го множеств в записи:
\[
\prd{E_V^m P_{j_1} \dots P_{j_{n-m}}}.
\]  
Совершая нужные повороты, придём к
\[
E_V^m\times  P_{j_1}\times \dots \times P_{j_{n-m}}.
\]

\textbf{Шаг 2.} Опишем как прийти к положению
\[
E_V\times E_0^{m-1}\times E_{j_1}\times \dots \times E_{j_{n-m}}.
\]
Двойной поворот заменяет два множества из произведения на противоположные, сохраняя порядок. Если необходимо поменять $i$-тое ($i > 1$) множество на противоположное
\[
P_j=E_{k-j} \text{ на } E_j \text{ или } E_V=E_{k-1} \text{ на } E_0,
\]
то реализуем $\psi^2_{1,i}$. Таким образом удастся 'настроить' все множества, кроме первого. Оставшееся принимает одно из двух значений: $E_0$ или $E_{k-1}$. Если  $E_V=E_0$, то мы пришли к $P$. Предположим противное.

\textbf{Шаг 3 (условие не выполнено).}

Отрицание данного условия даёт три условия на число $m$ и характеристики. Запишем каждое из них и нужную комбинацию поворотов, меняющую $E_{k-1}$ на $E_0$. Все положения на этом шаге имеют упрощённый вид.

\begin{itemize}
    \item $m>1$.
    
    $\psi_{2,1}$ переводит $A_1 \times A_2=E_{k-1} \times E_0$ в $E_0 \times E_0$.
    \item $m=1$ и $j_{n-1}=\mdd$.
    
    $\psi^2_{1,n}$ переводит $ A_1 \times A_n=E_{k-1} \times E_\md$ в положение $E_0 \times E_\md$.
    \item $m=1$ и найдётся $i: j_i=j_{i+1}<\mdd$.
    
    Композиция
    \[
    \psi_{m+i+1,m+i} \circ \psi^2_{1,m+i}
    \]
    переводит $A_1 \times A_{m+i} \times A_{m+i+1}=E_{k-1}\times E_{j_i} \times E_{j_i}$ в $E_{0}\times E_{k - j_i} \times E_{j_i}$, а потом в $E_{0}\times E_{j_i} \times E_{j_i}$.
\end{itemize}

\textbf{Шаг 3 (условие выполнено).} 
Заметим, что в данном случае никакие два множества положения не совпадают, с точностью до применения операции замены на противоположное. Иначе говоря, если записать положение как
\[
A_1 \times \dots \times A_n,
\]
то
\[
\forall i \neq j \; (A_i \neq A_j \: \text{и} \: A_i \neq \overline{A_j}).
\]
Значит, при переходе от одного положения к другому всегда можно сказать, сколько множеств стали противоположными. 

Каждый поворот совершает транспозицию двух множеств и заменяет одно на противоположное. То есть, чётность перестановки множеств в положении и чётность количества противоположных меняются одновременно. Но
\[
P^{'}=E_{k-1} \times E_{j_1}\times \dots \times E_{j_{n-1}}
\]
отличается от $P$ лишь чётностью противоположных. Значит, оно не переводится в $P$.

Множество разбивается на те элементы, которые переводятся в $P$, и которые переводятся в $P^{'}$. Каждое из них, очевидно, образует класс эквивалентности.

\textbf{Особые классы существуют}, если найдутся числа, удовлетворяющие неравенствам:
\[
1 \leq j_1 < j_2 <\dots < j_{n-1} < \md.
\]
Так как все $j_i$ натуральные, то $n-1 \leq j_{n-1}$ и $2n \le k$.
\end{proof}

Особые классы будем обозначать как $\spec$, где $q=1$, если кубики приводятся к $P$, и $q=-1$, если к $P^{'}$.

Примеры того, как кубики разбиваются на классы, приведён на рисунке \ref{sec3sub2pic1}.

\begin{deff}
Тривиальные классы $(n), \; (n-1)$ и $(n-t)$, где
\[
2\leq t \leq n-1,
\]
называются \textbf{угловыми, рёберными и $t$-гранными} соответственно.

Классы $(m,\md, \dots, \md)$ для $1 \leq m \leq n$ будем называть \textbf{центральными}. Название обусловлено тем, что они находятся в центрах внешних слоёв (см. покрашенные кубики на рис. \ref{sec3sub2pic1}).

Класс $(1,\md, \dots, \md)$ называется \textbf{каркасом куба}, ему мы уделим время при описании полной системы инвариантов.
\end{deff}

\begin{figure}[h]
    \centering
    \includegraphics[width=0.8\linewidth]{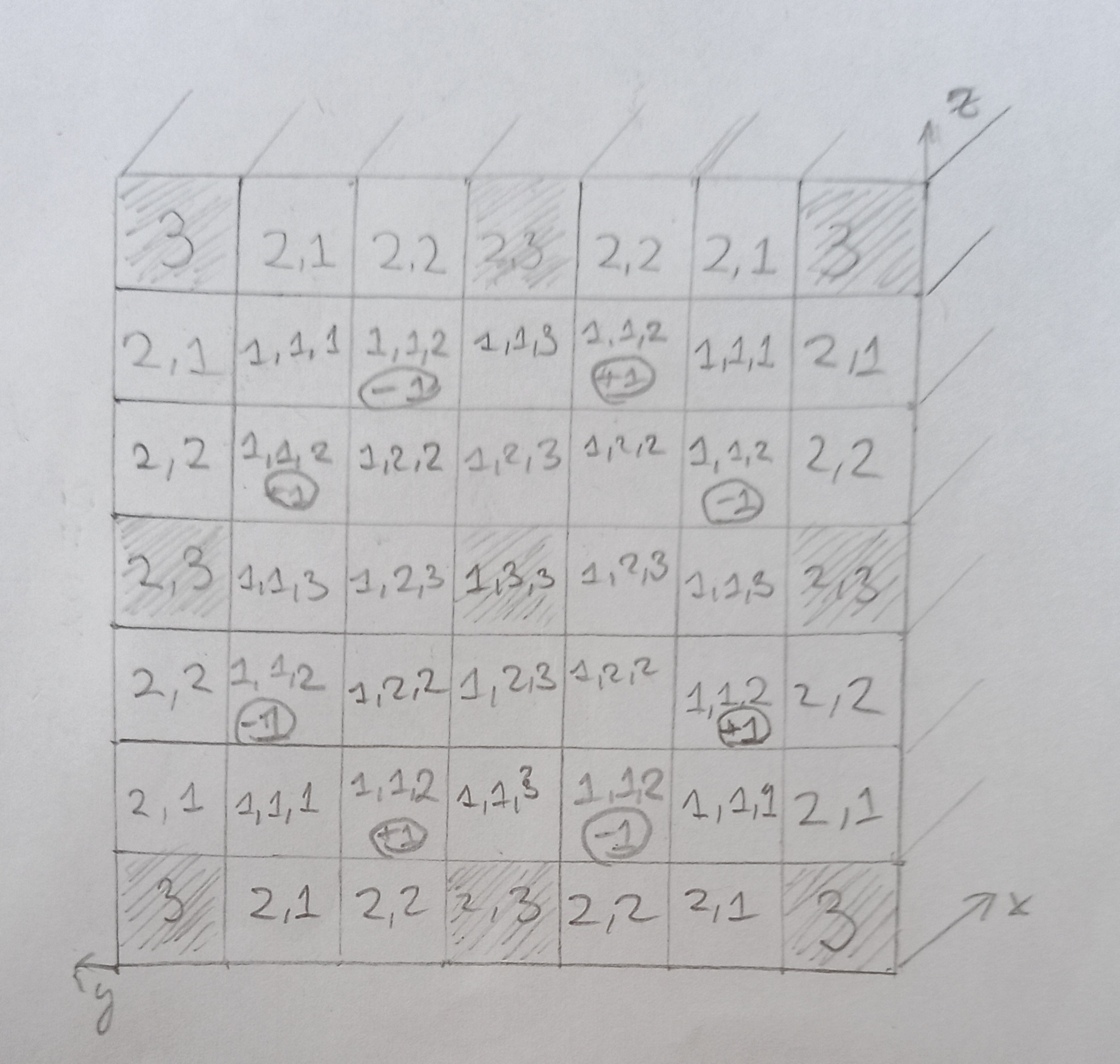}
    \caption{Классы куба $7 \times 7 \times 7$}
    \label{sec3sub2pic1}
\end{figure}

\begin{remark}
В этой работе мы не станем изучать внутренние кубики $E_F^n$, так как они не пересекаются ни с одной гранью куба и не обладают ни одной гранью.
\end{remark}

\subsection{Ориентационные грани}

\begin{lemma}
Предположим кубик обладает $t$-гранью, образованной пересечением с гранью куба
\[
\prd{V^{n-t} E^t}=\prd{V_1 \dots V_{n-t} E^t}.
\]
Тогда положение грани кубика задаётся как
\[
\prd{V^{n-t} E_M^t} =\prd{V_1 \dots V_{n-t} E_{j_1} \dots E_{j_{t}}}
\]
для некоторых чисел $j_1,\dots j_t \in M$. При этом положение самого кубика запишется в виде:
\[
\prd{E_V^{n-t} E_M^t}=\prd{E_{V_1} \dots E_{V_{n-t}} E_{j_1} \dots E_{j_t}}.
\]
(индексы при множествах $V$ и $E_M$ введены, чтобы подчеркнуть, какие из них в $\prd{V^{n-t} E_M^t}$ и $\prd{E_V^{n-t} E_M^t}$ совпадают).
\end{lemma}

\begin{proof}
Для начала перепишем положение кубика в удобном виде:
\[
E_M^n=\prd{E_M^n}=\prd{E_M^{n-t} E_{j_1} \dots E_{j_t}},
\]
где $j_1,\dots j_t$ -- некоторые числа из $M$, а $\sigma$ берём такую же, как у $\prd{V^{n-t} E^t}$.

В виду правил:
\[
X_1 \times \dots \times X_n \cap Y_1 \times \dots \times Y_n = (X_1 \cap Y_1) \times \dots \times (X_n \cap Y_n)
\]
\[
\prd{X_1 \dots X_n} \cap \prd{Y_1 \dots Y_n}=\prd{(X_1 \cap Y_1) \dots (X_n \cap Y_n)}
\]
имеем цепочку равенств:
\[
E_M^n \cap \prd{V^{n-t} E^t}=\prd{E_M^{n-t} E_{j_1} \dots E_{j_t}} \cap \prd{V^{n-t} E^t} =
\]
\[
=\prd{(E_M \cap V)^{n-t} (E_{j_1} \cap E) \dots (E_{j_t} \cap E)} =
\]
\[
= \prd{(E_M \cap V_1) \dots (E_M \cap V_{n-t}) E_{j_1} \dots E_{j_t}}.
\]
Пересечение не должно быть пустым. В то же время
\[
E_M \cap V_i = \begin{cases}
\varnothing, & E_M \neq E_{V_i} \\
V_i, & E_M = E_{V_i}
\end{cases}
\]
Так что положение $t$-мерной грани задаётся как
\[
\prd{V^{n-t} E_M^t} =\prd{V_1 \dots V_{n-t} E_{j_1} \dots E_{j_t}},
\]
а кубика:
\[
\prd{E_V^{n-t} E_M^t}=\prd{E_{V_1} \dots E_{V_{n-t}} E_{j_1} \dots E_{j_t}}.
\]
\end{proof}

\begin{cons}
У кубика тривиального класса $(m)$ всего $\binom{m}{n-t} \; t$-граней.
\end{cons}

\begin{proof}
По определению положение кубика будет
\[
E_M^n=\prd{E_V^{m} E_F^{n-m}}.
\]
Из предыдущей леммы видно, что положения $t$-граней получаются из положения кубика заменой $(n-t)$ множеств $E_{V_i}$ на $V_i$. Тогда искомое количество граней -- это количество способов выбрать из $m$ множеств $(n-t)$ штук:
\[
\binom{m}{n-t}
\]
\end{proof}

В частности, у кубика тривиального класса $(m)$
\begin{itemize}
    \item нет $t$-граней, если $t<n-m$;
    \item одна $(n-m)$-мерных грань;
    \item $m$ $(n+1-m)$-мерных граней.
\end{itemize}

Временно определим ориентацию кубика как совокупность положений \textbf{всех} его покрашенных граней. Как показывает следующая теорема, это условие излишне, небольшого количества граней вполне хватает для определения ориентации.

\begin{theorem} \label{orientation}
Все $(n+1-m)$-мерные грани кубика класса $\class$ однозначно определяют его ориентацию (Далее их называем \textbf{ориентационными}). Иначе говоря, зная положения этих граней, можно указать положения всех остальных.
\end{theorem}

\begin{proof}
Переформулируем утверждение, которое нужно доказать. Пусть кубик находится в фиксированном положении и мы меняем его ориентацию. Тогда нельзя изменить положение $t$-грани ($t>(n+1-m)$), не меняя положений ориентационных.

Достаточно доказать это утверждение только для одного из кубиков класса, так что выберем тот, который находится в каноническом положении
\[
E_0^m \times E_{j_1} \times \dots \times E_{j_{n-m}}.
\]
Из предыдущей леммы следует, что для получение положения какой-либо его $(n+1-m)$-грани необходимо $(m-1)$ множеств $E_V=E_0$ заменить на $V=\{0\}$.

Выберем произвольную грань кубика размерности большей, чем $(n+1-m)$. Докажем, что изменение её положения влечёт изменения положений ориентационных. Для этого запишем положения выбранной и ориентационных граней в таблицу по строчкам следующим образом и пронумеруем сами грани

\begin{center}
\begin{tabular}{c c c c c c c c c}
 $(0)$ & $E_0$ & $V$ & $E_0$ & \dots & $V$ & $E_{j_1}$ & \dots & $E_{j_{n-m}}$ \\
 $1$ & $E_0$ & $V$ & $V$ & \dots & $V$ & $E_{j_1}$ & \dots & $E_{j_{n-m}}$ \\
 $2$ & $V$ & $E_0$ & $V$ & \dots & $V$ & $E_{j_1}$ & \dots & $E_{j_{n-m}}$ \\ 
 $3$ & $V$ & $V$ & $E_0$ & \dots & $V$ & $E_{j_1}$ & \dots & $E_{j_{n-m}}$ \\
 \dots &  \dots &  \dots &  \dots & \dots &  \dots &  \dots & \dots & \dots \\
 $m$ & $V$ & $V$ & $V$ & \dots & $E_0$ & $E_{j_1}$ & \dots & $E_{j_{n-m}}$.
\end{tabular}
\end{center}

Пусть в положении выбранной грани множества $E_0$ находятся на $i_1, \dots, i_{l}$ позициях. Легко увидеть из таблицы, что тогда такая грань содержит в качестве подмножест только $i_1, \dots, i_{l}$ ориентационные. Подобное свойство, конечно, не должно зависеть от положения и ориентации кубика.

Любое изменение положения выбранной грани влечёт за собой иную расстановку $l$ множеств $E_0$ в записи положения. Но тогда и уже другие ориентационные грани будут содержаться в выбранной, что невозможно согласно замечанию выше.
\end{proof}

\begin{remark}
Обратим внимание, что для угловых есть $n$ ориентационных граней размерности 1. Они не являются обычными гранями в смысле нашего определения и не покрашены в какие-либо цвета. Тем не менее, они однозначно определяют ориентации угловых элементов головоломки.
\end{remark}

На схематическом рисунке \ref{sec3sub3pic1} изображен четырёхмерный куб и приведён пример того, какие ориентационные грани есть у его кубиков

\begin{itemize}
    \item У угловых -- 4 одномерные;
    \item У рёберных -- 3 двумерные;
    \item У двугранных -- 2 трёхмерные;
    \item У трёхгранных -- 1 четырёхмерная.
\end{itemize}

\begin{figure}[h]
    \centering
    \includegraphics[width=0.7\linewidth]{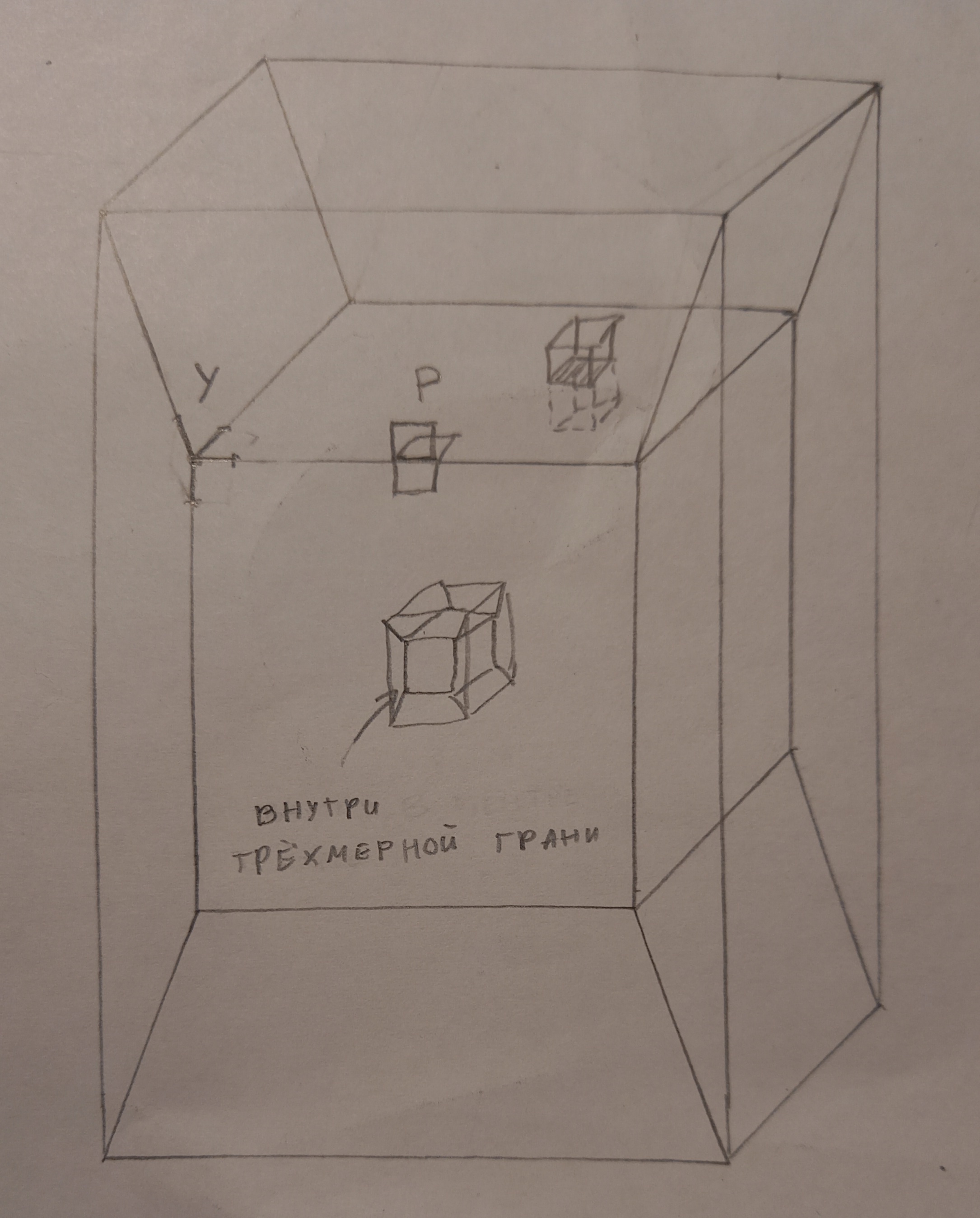}
    \caption{}
    \label{sec3sub3pic1}
\end{figure}

\begin{deff}
Теперь под \textbf{ориентацией} кубика будем понимать совокупность положений всех его ориентационных граней.

\textbf{Состояние кубика} -- это его положение и ориентация.

\textbf{Положением (ориентацией) множества} называется совокупность положений (ориентаций) всех его кубиков.

\textbf{Состоянием множества} называем его положение и ориентацию.

Два состояния множества кубиков называются $\textbf{связанными}$, если можно от одного из них прийти к другому поворотами слоёв Кубика Рубика.
\end{deff}

\subsection{Группы вращения}

Для каждого кубика из $\class, \; m>1$ можно определить \textbf{две группы вращения}, являющиеся подгруппами $S_m$ (группы перестановок $m$ ориентационных граней)

\begin{itemize}
    \item Рассмотрим те самосовмещения кубика в пространстве, при которых ориентационные грани переходят в ориентационные. Все перестановки этих граней, порождённые такими поворотами, образуют \textbf{независимую группу вращения} $\Gr$.
    
    Кубики из одного тривиального класса отличаются лишь покраской граней и расположением в головоломке. Взаимные расположения их ориентационных граней одни и те же, так что данная группа зависит от $m$ и \textbf{не зависит от характеристик класса}.
    
    \item Говорят, что комбинация поворотов \textbf{сохраняет положение кубика}, если она возвращает его на исходную позицию, возможно изменив в процессе его ориентацию и состояния остальных элементов головоломки. При такой комбинации ориентационные грани переходят, конечно, в ориентационные.
    
    Все перестановки ориентационных граней, порождённые сохраняющими положение комбинациями, образуют \textbf{зависимую группу вращения} $\oG$. Данная группа уже \textbf{зависит и от $m$ и от характеристик класса}, так как для вышеприведённых комбинаций поворотов положение играет существенную роль. Зависимая группа является подгруппой независимой.
\end{itemize}

Например, комбинация $RU$ (сначала поворот $R$, затем -- $U$) возвращает угловой кубик трёхмерного Кубика Рубика на свое место и меняет циклически три его ориентационные грани (рис. \ref{sec3sub4pic1}).

\begin{figure}[h]
    \centering
    \includegraphics[width=0.6\linewidth]{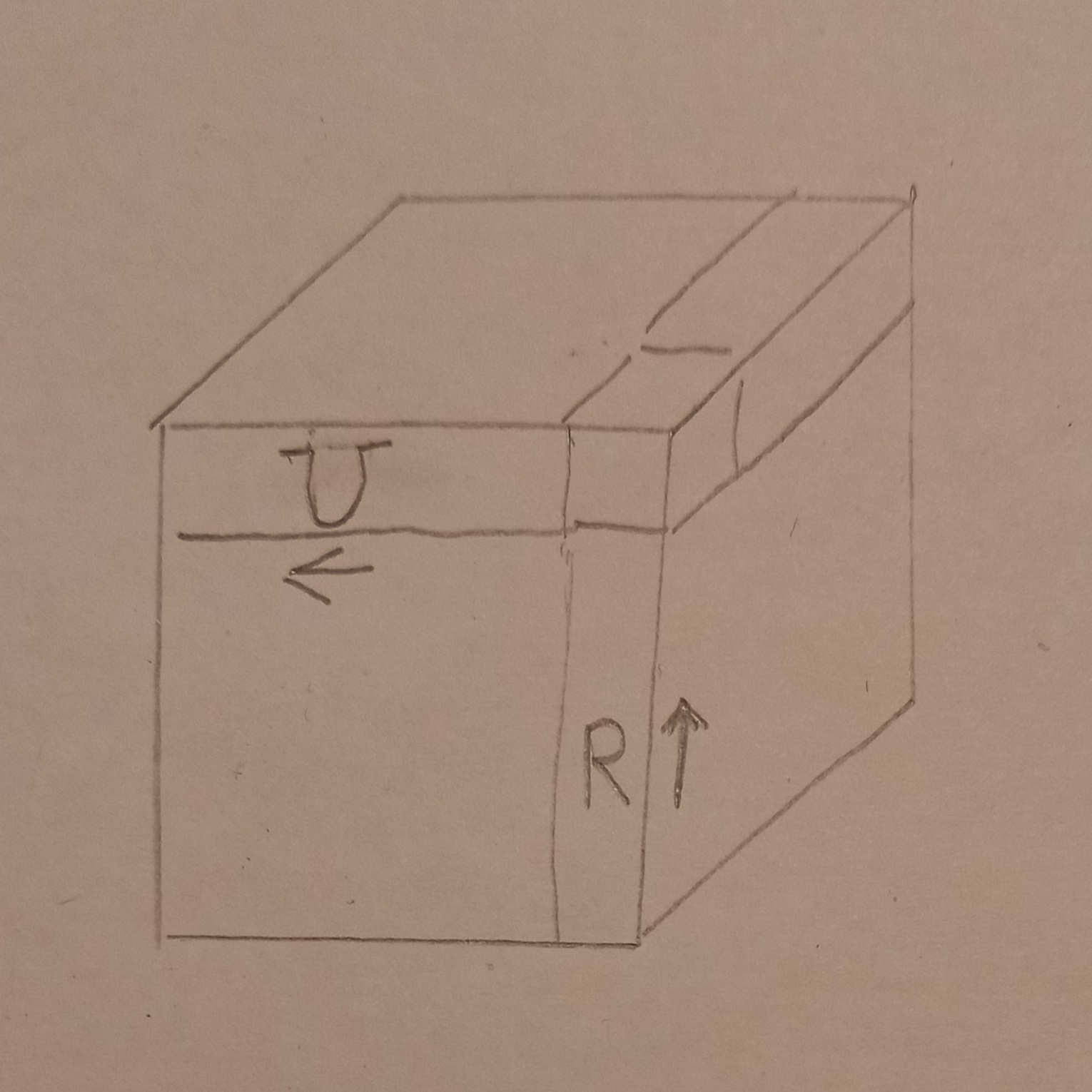}
    \caption{}
    \label{sec3sub4pic1}
\end{figure}

Проделав её ещё раз, получаем уже другую циклическую перестановку. В итоге зависимая группа заведомо содержит $A_3$, изоморфную $\Z_3$. Легко понять, что две грани углового кубика никак не могут поменяться местами, так что $\Gr = \oG \cong Z_3$.

В общем случае эти две группы различны. Например, возьмём нецентрально рёберный элемент трёхмерного Кубика Рубика. Его независимая группа, конечно, изоморфна $Z_2$, но зависимая -- $\{e\}$. То есть, в фиксированном положении такого кубика может быть ровно одна ориентация. Объяснение этого эффекта будет дано в следующей теореме.

У каждого кубика из (1) только одна ориентационная грань, так что говорить об их ориентации быссмысленно.

\begin{twogroups}
Для элемента из $(m)$:
\[
G_{rot} (m)=\begin{cases}
A_m, & m=n, \\
S_m, & m<n.
\end{cases}
\]
Для элемента из $\class$:
\[
\overline{G}_{rot} \class =\begin{cases}
A_m, & m=n \text{ или } j_1 < \dots < j_{n-m} < \md, \\
S_m, & m<n \text{ и } (j_{n-m}=\md \text{ или } \exists p: j_p=j_{p+1}<\md).
\end{cases}
\]

\end{twogroups}

\begin{proof}
\textbf{Зависимая группа.}

Поступим как в теореме \ref{orientation}: выберем класс, возьмём кубик в каноническом положении и докажем теорему для него.

Изобразим таблицу ориентационных граней:

\begin{center}
\begin{tabular}{c c c c c c c c c}
 1 & $E_0$ & $V$ & $V$ & \dots & $V$ & $P_{j_1}$ & \dots & $P_{j_{n-m}}$ \\
 2 & $V$ & $E_0$ & $V$ & \dots & $V$ & $P_{j_1}$ & \dots & $P_{j_{n-m}}$ \\
 3 & $V$ & $V$ & $E_0$ & \dots & $V$ & $P_{j_1}$ & \dots & $P_{j_{n-m}}$ \\
 \dots & \dots & \dots & \dots & \dots & \dots & \dots & \dots & \dots \\ 
 $m$ & $V$ & $V$ & $V$ & \dots & $E_0$ & $P_{j_1}$ & \dots & $P_{j_{n-m}}$.
\end{tabular}
\end{center}
Опять же, для удобства, вместо $\{0\}$ пишем $V$.

\textbf{1. Циклические перестановки} могут быть осуществимы, если найдутся три ориентационные грани, то есть при $m \geq 3$.

Комбинация $\psi_{s,i} \; \circ \; \psi_{i,j}$, где $1 \leq i<j<s \leq m$, сохраняет положение и циклически меняет три ориентационные грани: $i$-ую, $j$-ую и $s$-ую.

Действительно, в упрощённой записи получаем
\[
(\psi_{s,i} \; \circ \; \psi_{i,j}) (E_0 \times E_0 \times E_0) = \psi_{s,i} (E_{k-1} \times E_0 \times E_0) = E_0 \times E_0 \times E_0.
\]

Что касается граней, каждый поворот меняет местами два столбца, а затем в одном из них все элементы -- на противоположные. Если $l$-ый столбец обозначить за $C_l$, то аналогично положению
\[
(\psi_{s,i} \; \circ \; \psi_{i,j}) (C_i,C_j,C_s) = (C_s, C_i, C_j).
\]
При такой смене столбцов, окажется, что новая таблица отличается от старой лишь перестановкой трёх строк, отвечающих $i$-ой, $j$-ой и $s$-ой граням.

Итак, зависимая группа содержит $A_m$.

\textbf{2.} Выясним, когда осуществимы \textbf{транспозиции} ориентационных граней.

\textbf{Пусть соблюдается условие}
\[
m=n \text{ или } j_1 < \dots < j_{n-m} < \md.
\]
Тогда все столбцы различны, с точностью до замены их на противоположные. По той же причине, которая была указана в теореме о классификации, невозможно осуществить транспозицию столбцов. Тогда $\oG=A_m$.

\textbf{Если оно не соблюдается, то} $m<n$ и выполняется одно из двух условий на характеристики. Для каждого из них выпишем комбинацию, реализующую транспозицию столбцов. Тогда, как и в предыдущем пункте, новая таблица будет отличаться от старой перестановкой двух строк.

\begin{itemize}
    \item 
    \[
    j_{n-m}=\md.
    \]
    Тогда для последнего столбца справедливо: $\overline{C}_m=C_m$.
    
    Комбинация $\psi^2_{i,m} \circ \psi_{i,j}$ совершает транспозицию столбцов $i$ и $j$:
    \[
    (\psi^2_{i,m} \circ \psi_{i,j}) (C_i,C_j,C_{m}) = \psi^2_{i,m} (\overline{C}_j,C_i,C_{m})=(C_j,C_i,\overline{C}_{m})=(C_j,C_i,C_{m}).
    \]
    \item
    \[
    \exists p: j_p=j_{p+1} < \md.
    \]
    Тогда $C_{m+p}=C_{m+p+1}$. Также из условия следует, что характеристик как минимум две, а $m \geq 2$, и, следовательно, $n \geq m+2 = 4$. То есть, можно пользоваться поворотами в четырёхмерном пространстве.
    
    $\psi_{m+p+1,m+p} \circ \psi^2_{i,m+p} \circ \psi_{i,j}$ совершает транспозицию столбцов $i$ и $j$:
    \[
    (\psi_{m+p+1,m+p} \; \circ \; \psi^2_{i,m+p} \; \circ \; \psi_{i,j}) (C_i,C_j,C_{m+p},C_{m+p+1}) =
    \]
    \[
    = \psi_{m+p+1,m+p} (C_j,C_i,\overline{C}_{m+p},C_{m+p+1}) = (C_j,C_i,C_{m+p},C_{m+p+1}).
    \]
\end{itemize}
Значит, $\oG=S_m$.

\textbf{Независимая группа.}

\textbf{Пусть $m<n$.} Легко убедиться, что $\Gr$ не зависит не только от характеристик, но и от параметра $k$. Он влияет на расположение кубиков в пространстве, но не на количество и взаимное расположение ориентационных граней.

Из этих соображений следует, что $\Gr (m)$ содержит зависимую группу $\overline{G}_{rot} \class$ при любых характеристиках и $k$. Выберем нечётное $k$ и $j_{n-m}=\mdd$. Тогда $\Gr$ содержит $S_m=\overline{G}_{rot} \class$ и, следовательно, совпадает с ней. 

\textbf{Пусть $m=n$ (угловые).} Осуществить транспозицию двух ориентационных граней невозможно. Пойдём от противного. Пусть, например, первая и вторая ориентационные грани поменялись местами.

Каждая ориентационная грань одномерна и является отрезком. Все такие отрезки исходят из одной вершины. Выберем какую-либо их тройку, содержащую первый и второй отрезки. Их транспозиция приводит к смене ориентации тройки: левая тройка становится правой, и наоборот. Это, конечно, не может случиться при вращении в пространстве.
\end{proof}

\begin{deff}
Будем называть класс \textbf{симметрическим (знакопеременным)}, если зависимая группа вращения его элементов $S_m$ (соответственно, $A_m$).
\end{deff}

\subsection{Ориентационная перестановка}
Временно забудем про цвета и пронумеруем все ориентационные грани каждого кубика из $(m)$ числами от 1 до $m$. В свою очередь, каждому положению в головоломке сопоставим \textbf{правильную ориентацию}: то, как должны располагаться пронумерованные грани. Заметим, что мы пока не требует связанности кубика текущей ориентацией и кубика правильной ориентацией.

\begin{deff}
Для некоторого кубика из $(m)$ за $i_j$ обозначим номер грани в текущей ориентации, вместо которой должна быть $j$-ая согласно правильной. \textbf{Ориентационной перестановкой} $\phi \in S_m$ называется перестановка, действующей по правилу
\[
\phi(j)=i_j.
\]
При правильной ориентации перестановка тождественна.
\end{deff}

\begin{orientationlemma}
Пусть совершён поворот $\psi_{ij}$ произвольного слоя. 

1. Если положение кубика не меняется, то и его ориентационная перестановка тоже.

2. Если он циклически переставляется с другими, то пронумеруем их все так, чтобы первый переходил в положение второго, \dots, четвёртый -- в положение первого.

Обозначим за $\alpha_l$ ориентационную перестановку $l$-го кубика после поворота, если бы изначально его ориентация была правильной. $\alpha_l$ будем кратко называть \textbf{переходом}. $\phi_l$ и $\phi_l^{'}$ пусть будут его старой и новой перестановками.

Тогда справедливо:
\[
\phi_l^{'}=\phi_l \alpha_l,
\]
\[
\alpha_4 \alpha_3 \alpha_2 \alpha_1 = e.
\]

\end{orientationlemma}

\begin{proof}
1. Неизменность положения говорит о том, что его запись в упрощённом виде
\[
E_\md \times E_\md.
\]
Тогда в таблице ориентационных граней $i$-ый и $j$-ый столбцы будут состоять из множеств $E_\md$. Следовательно, поворот $\psi_{ij}$ не изменит таблицу и положений граней.

2. Выберем один из четырёх кубиков с переходом $\alpha$. Пусть $\alpha(i)=j$. Тогда грань $j$ кубика правильной ориентации перемещается туда, где должна быть грань $i$ в новой правильной ориентации (см. рис. \ref{sec3sub5pic1}). Значит, $\phi(j)$ перемещается на тоже самое место и
\[
\phi^{'}(i)=\phi(j)=\phi(\alpha(i)).
\]
Четырёхкратный поворот -- тождественный поворот и, следовательно,
\[
\phi_1 \alpha_4 \alpha_3 \alpha_2 \alpha_1 = \phi_1.
\]
для любой $\phi_1$, в частности, и для $\phi_1=e$.
\end{proof}

\begin{figure}[h]
    \centering
    \includegraphics[width=0.8\linewidth]{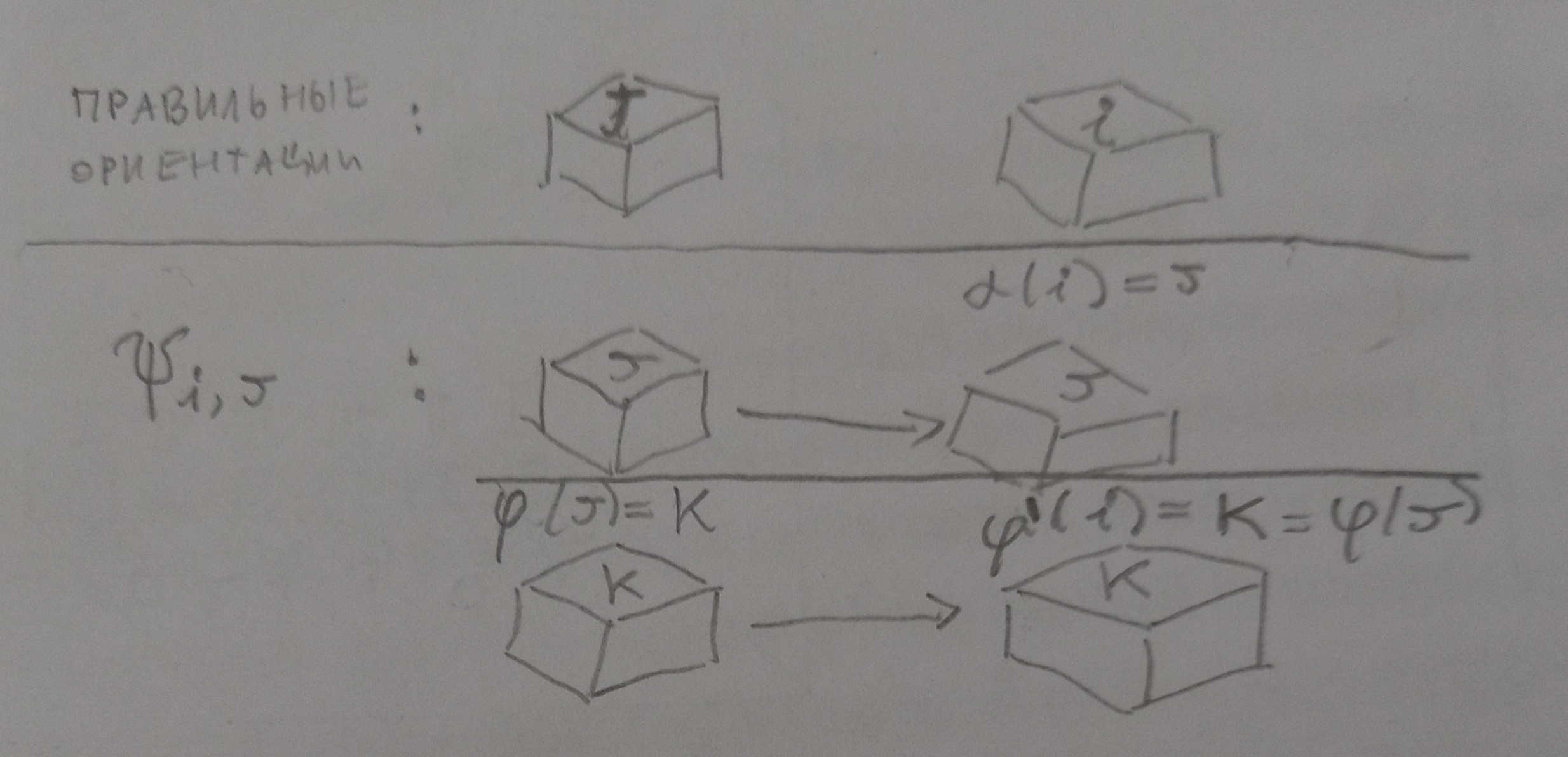}
    \caption{}
    \label{sec3sub5pic1}
\end{figure}

$\phi$ зависит от способа нумерации граней и выбора правильных ориентаций. Очевидно, что с ней было бы легче работать, будь она перестановкой из зависимой группы. Иными словами, если бы каждая правильная ориентация была достижима из нынешней ориентации.

\begin{lemma}
Существует способ нумерации граней и выбора правильных ориентаций в Кубике Рубика, при котором все ориентационные перестановки -- элементы зависимой группы. Более того, это свойство сохраняется при вращении слоёв головоломки.
\end{lemma}

\begin{proof}
Выберем произвольный класс и докажем для него.

\textbf{Нумерация граней и выбор правильных ориентаций.}

Выберем произвольную правильную ориентацию для $P$.

Пронумеруем все положения. Кубик в положении $i$ переместим в $P$ некоторой комбинацией $q_i$. Пронумеруем грани так, чтобы ориентация была правильной, и вернём на свое место комбинацией $q^{-1}_i$. Проделаем так для каждого $i$.

Итак, все грани пронумерованы. Осталось сопоставить положениям ориентации так, что все ориентации кубиков стали правильными. Тогда ориентационные перестановки тождественны и принадлежат зависимой группе.

\textbf{Инвариантность относительно вращения слоёв.}

Докажем, что все переходы в головоломке являются элементами зависимой группы. Тогда все новые ориентации также будут принадлежать $\oG$ согласно лемме об ориентации.

Обозначим за $\psi$ поворот, переводящий кубик положения $i$ в $j$, а за $\alpha$ соответствующий переход.

Комбинация поворотов $q_j q^{-1}_i \psi$ переводит
\begin{itemize}
    \item кубик в $P$, $\phi_P=e$ согласно выбору $q_i$ и ориентаций.
    \item затем в $i$-ое, где $\phi_i=e$ снова.
    \item в $j$ поворотом $\psi$, где $\phi_j = \alpha$.
\end{itemize}
Так как новая ориентация достижима из старой, то $\alpha=\phi_j \in \oG$.
\end{proof}

\subsection{Кластеры}

\begin{deff}
Назовём кубик \textbf{уникальным}, если нет других кубиков того же класса с такой же раскраской граней.

\textbf{Кластер} -- это набор всех кубиков одного класса с одинаковой раскраской. В таком наборе должно быть больше одного кубика. Иначе говоря, множество из одного уникального кубика кластером не считается.

\textbf{Тривиальным кластером} называем кластер кубиков, у которых одна ориентационная грань, то есть они из тривиального класса $(1)$.

\textbf{Знакопеременный кластер} -- кластер, группа вращения кубиков которого знакопеременна. Аналогично определяется \textbf{симметрический}. Понятно, что знакопеременный класс содержит только знакопеременные кластеры, а симметрический -- симметрические.
\end{deff}

На рисунке \ref{sec3sub6pic1} приведены примеры в трёхмерном кубе: фиолетовым отмечен нетривиальный кластер, голубым -- тривиальный. Все нетривиальные кластеры в трёхмерном кубе ограничиваются парами нецентрально-рёберных кубиков, в 'многомерье' же их видов намного больше.

\begin{figure}[h]
    \centering
    \includegraphics[width=0.5\textwidth]{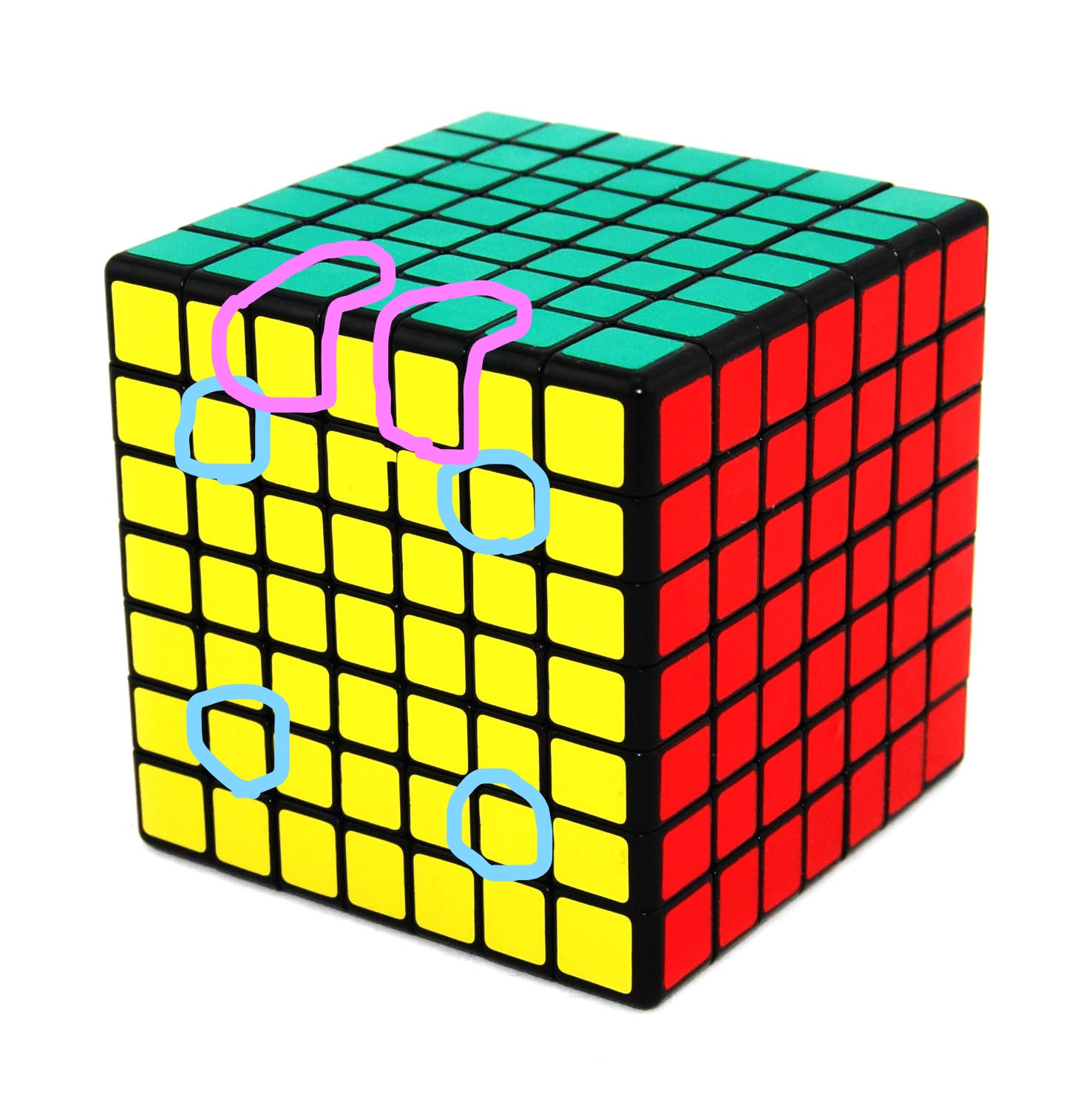}
    \caption{}
    \label{sec3sub5pic1}
\end{figure}

Таким образом, иерархия элементов головоломки такая:

\begin{itemize}
    \item все кубики делятся на тривиальные классы
    \item тривиальные классы -- на классы
    \item классы -- на кластеры или на уникальные кубики.
\end{itemize}

Ранее были описаны все возникающие классы и тривиальные классы. Теперь приступим к кластеров.

\begin{cluster}
Пусть задан класс $\class$. Он распадается на подмножества
\[
Cl(\sigma,V_1,\dots,V_m,j_1,\dots,j_{n-m}),
\]
каждое из которых состоит из всех кубиков с положениями
\[
\prd{E_{V_1} \dots E_{V_m} P_{k_1} \dots P_{k_{n-m}}}.
\]
Здесь, конечно, $\sigma, V_1,\dots,V_m$ и $j_1,\dots,j_{n-m}$ зафиксированы, а $k_1,\dots,k_{n-m}$ -- это характеристики $j_1,\dots,j_{n-m}$, записанные в произвольном порядке. То есть, положения кубиков из $Cl$ отличаются лишь перестановками множеств $P_j$.

Что касается особого класса $\spec$, он разбивается на
\[
Cl(\sigma,V_1,\dots,V_m,j_1,\dots,j_{n-m},q).
\]
Каждое состоит из кубиков подмножества
\[
Cl(\sigma,V_1,\dots,V_m,j_1,\dots,j_{n-m}),
\]
связанных с $P$, если $q=1$, и не связанных, если $q=-1$.

Если выбранный класс не является центральным, то каждое такое подмножество $Cl$ является кластером. Иначе, каждый $Cl$ состоит из одного уникального кубика.
\end{cluster}

\begin{proof}
\textbf{1.} Выберем класс $\class$ и докажем утверждение для него. Положения его кубиков, как мы знаем
\[
\prd{E_{V_1} \dots E_{V_m} P_{j_1} \dots P_{j_{n-m}}}.
\]
Легко увидеть, что каждый кубик класса попадает в некоторое
\[
Cl(\sigma,V_1,\dots,V_m,j_1,\dots,j_{n-m}),
\]
а сами подмножества не пересекаются.

\textbf{2.} Если класс является центральным, то каждое $Cl$ состоит всего из одного кубика
\[
\prd{E_{V_1} \dots E_{V_m} E^{n-m}_{\md}}.
\]
Если же нет, то минимальная характеристика $j_1$ отлична от $\mdd$ и множество $P_{j_1}$ может принимать одно из двух значений: $E_{j_1}$ или $E_{k-j_1}$. Тогда в подмножестве найдутся хотя бы два элемента.

\textbf{3.} Осталось показать, что у каждого $Cl$ своя раскраска.

Для определённости возьмём 
\[
Cl(e,V_1,\dots,V_m,j_1,\dots,j_{n-m}),
\]
положения кубиков которого
\[
E_{V_1} \times \dots \times E_{V_m} \times P_{k_1} \times \dots \times P_{k_{n-m}}.
\]
Каждый пересекается со следующими $m$ $(n+1-m)$-мерными гранями куба.
\begin{center}
\begin{tabular}{c c c c c c c}
 $E_{V_1}$ & $V_2$ & \dots & $V_m$ & $E$ & \dots & $E$ \\
 $V_1$ & $E_{V_2}$ & \dots & $V_m$ & $E$ & \dots & $E$ \\
 $V_1$ & $V_2$ & \dots & $V_m$ & $E$ & \dots & $E$ \\
 \dots & \dots & \dots & \dots & \dots & \dots & \dots \\ 
 $V_1$ & $V_2$ & \dots & $E_{V_m}$ & $E$ & \dots & $E$.
\end{tabular}
\end{center}
То, в какие цвета окрашены эти грани, однозначно определяет в какие цвета окрашены кубики.

Заметим, что множество всех кубиков из выбранного класса, которые пересекаются с данными $(n+1-m)$-мерными гранями, должно совпасть с
\[
Cl(e,V_1,\dots,V_m,j_1,\dots,j_{n-m}).
\]
Действительно, первые $m$ множеств обязаны быть $E_{V_1} \times \dots \times E_{V_m}$. Остальные могут произвольными, но так как кубики из класса $\class$, то они должны быть $P_{j_i}$.

Выходит, что кубики окрашены одинаково тогда и только тогда, когда они принадлежат одному подмножеству $Cl$.

Доказательство для особых классов проводится аналогично.
\end{proof}

\begin{cons}
Знакопеременный класс разбивается на $\binom{n}{m} \cdot 2^m$ знакопеременных кластеров.

Количество элементов в каждом из них
\[
Cl(\sigma,V_1,\dots,V_m,j_1,\dots,j_{n-m})
\]
одинаково и равно $(n-m)! \cdot 2^{n-m}$.
\end{cons}

\begin{proof}

\textbf{1.} Выберем произвольный кластер из класса, положения кубиков в нём
\[
\prd{E_{V_1} \dots E_{V_m} P_{k_1} \dots P_{k_{n-m}}}.
\]
Так как он знакопеременный, то $m < n$ и
\[
j_1 < \dots < j_{n-m} < \md.
\]
То есть все характеристики различны. Тогда можно реализовать
\[
(n-m)! \cdot 2^{n-m}
\]
положений, где
\begin{itemize}
    \item $(n-m)!$ -- количество перестановок множеств $P_{j_i}$.
    \item $2^{n-m}$ -- количество способов выбрать из каждого $P_{j_i}$ множество $E_{j_i}$ либо $E_{k-j_i}$.
\end{itemize}

\textbf{2.} Общее количество элементов в знакопеременном классе
\[
\frac{n!}{m!} \cdot 2^n.
\]
\begin{itemize}
    \item $\frac{n!}{m!}$ -- количество способов расстановки $n$ множеств, из которых $(n-m)$ различных $P_j$ и $m$ множеств $E_V$.
    \item $2^n$ -- множитель, связанный с тем, что из каждого $E_V$ и $P_j$ мы выбираем одно из двух множеств.
\end{itemize}
Осталось разделить количество элементов класса на количество элементов кластера и получить
\[
\binom{n}{m} \cdot 2^m.
\]
\end{proof}

Кластеры можно располагать на своём месте разными способами в виду неотличимости элементов. Этот факт существенен, когда речь идёт об ориентировании знакопеременных кластеров.

\begin{example}
На рисунке \ref{sec3sub6pic2} изображён знакопеременный кластер. На первый взгляд изменить его состояние невозможно, ведь, как было показано, зависимая группа его элементов состоит только из тождественной перестановки.

Однако стоит поменять местами кубики кластера, как его ориентация изменится. Внешнее сходство элементов позволило перейти от одного состояния к другому, которое казалось несвязанным с изначальным.
\end{example}

\begin{figure}[h]
    \centering
    \includegraphics[width=0.9\textwidth]{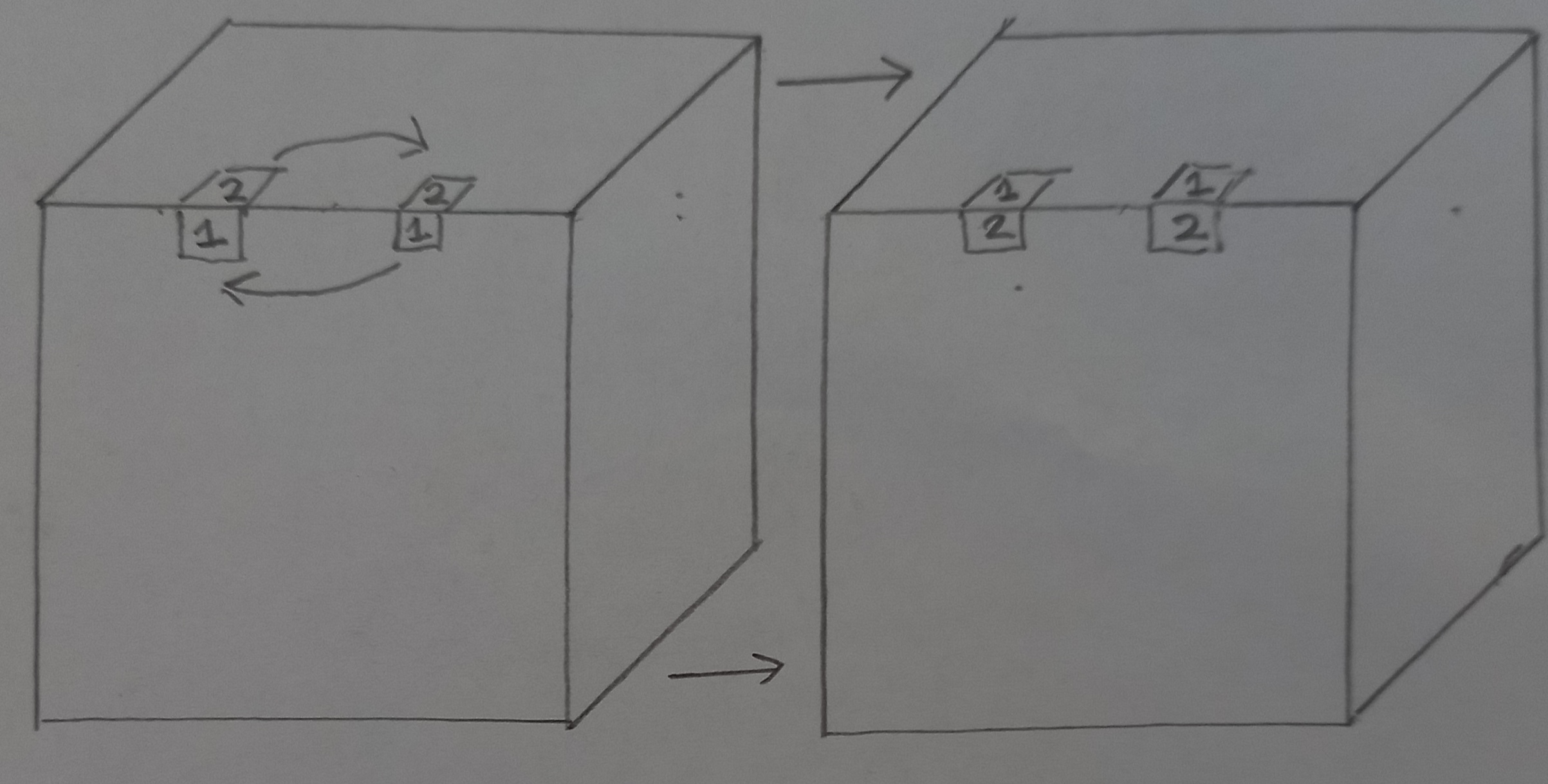}
    \caption{}
    \label{sec3sub6pic2}
\end{figure}

\section{Аккуратные комбинации}

\begin{deff}
$\textbf{Аккуратная перестановочная комбинация}$ совокупности кубиков -- это комбинация поворотов, изменяющая положения только её кубиков. При этом состояния остальных элементов Кубика Рубика не меняются. Такая комбинация называется $\textbf{чётной}$, если перестановка совокупности чётна.

$\textbf{Аккуратная ориентационная комбинация}$ совокупности изменяет ориентации её кубиков, не меняет их положений и состояний всех остальных элементов головоломки.

Действие комбинации на кубике \textbf{тривиально}, если его состояние не меняется.

$\textbf{Областью действия}$ комбинации считаем все положения кубиков, для которых действие данной комбинации нетривиально.
\end{deff}

В этом разделе все классы, о которых пойдёт речь, не являются каркасными. Последние будут подробно изучены позже. Также мы будем использовать два способа обозначения комбинаций. Первый уже был использован ранее и заключался в том, что комбинации обозначались прописными латинскими буквами.

В этом разделе будут приведены наглядные примеры перестановки и ориентирования кубиков. На рисунках повороты будут обозначены заглавными латинскими буквами, например, $R$. Также будет указано стрелкой, в какую сторону необходимо совершить этот поворот. Если в комбинации потребуется совершить поворот в противоположном направлении, то это будет обозначено как $R^{-1}$.

Прежде чем приступить к изучению аккуратных комбинаций, сформулируем следующее простое и полезное утверждение.

\begin{carefulcomb}
Пусть даны $n$ кубиков и

\begin{itemize}
    \item $q$ -- аккуратная перестановочная (ориентационная) комбинация для $n$ кубиков в головоломке, стоящих в определённых позициях.
    \item $p$ переводит $n$ кубиков из их исходных положений в область действия $q$.
\end{itemize}

Тогда комбинация
\[
w=pqp^{-1}
\]
аккуратно переставляет (ориентирует) выбранные кубики в их исходных положениях.
\end{carefulcomb}

Его доказательство очевидно.

\subsection{Перестановочные комбинации}

Сформулируем главную теорему перестановок в Кубике Рубика.

\begin{Classpermutation}
Можно осуществить любую аккуратную чётную перестановку произвольного класса.
\end{Classpermutation}

Доказательство теоремы будет мгновенно вытекать из следующих вспомогательных лемм.

\begin{lemma} \label{3dimpermutation}
В трёхмерном кубике Рубика можно осуществить любую аккуратную чётную перестановку произвольного класса.
\end{lemma}

\begin{proof}
Достаточно доказать, что осуществима любая циклическая перестановка трёх кубиков. Для этого рассмотрим четыре случая, и для каждого из них явно укажем искомую комбинацию. Проверка правильности этих комбинаций тривиальна.

Пусть три кубика из класса
\begin{itemize}
    \item угловых (рис. \ref{sec4sub1pic1})
    \item рёберных (рис. \ref{sec4sub1pic2})
    \item двугранных, которые расположены на диагоналях внешних слоёв (рис. \ref{sec4sub1pic3})
    \item двугранных, которые расположены вне диагоналей внешних слоёв (рис. \ref{sec4sub1pic4})
\end{itemize}
расположены как на соответствующих рисунках.

\begin{figure}[h]
\begin{center}
\begin{minipage}[h]{0.4\linewidth}
\includegraphics[width=1\linewidth]{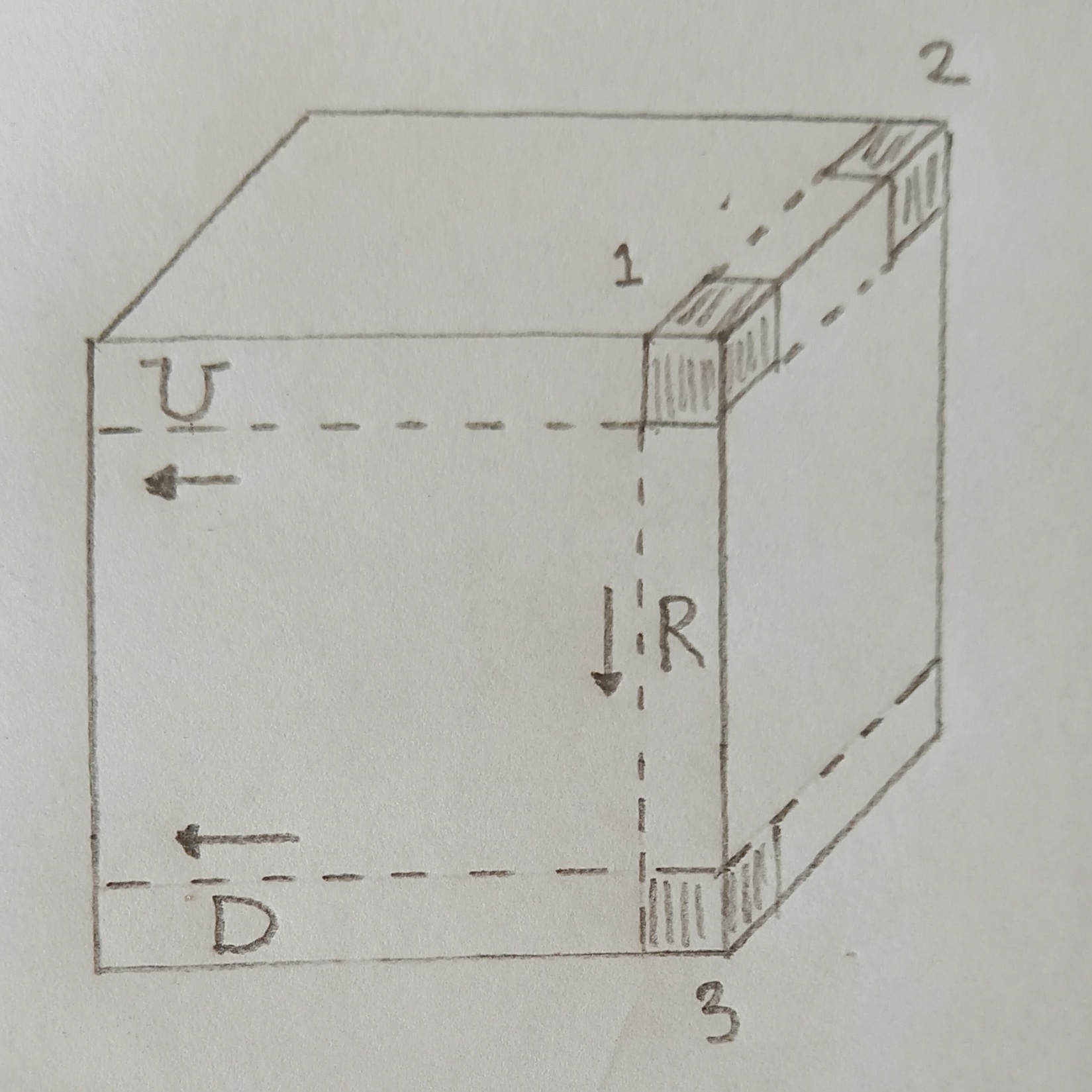}
\caption{}
\label{sec4sub1pic1}
\end{minipage}
\begin{minipage}[h]{0.4\linewidth}
\includegraphics[width=1\linewidth]{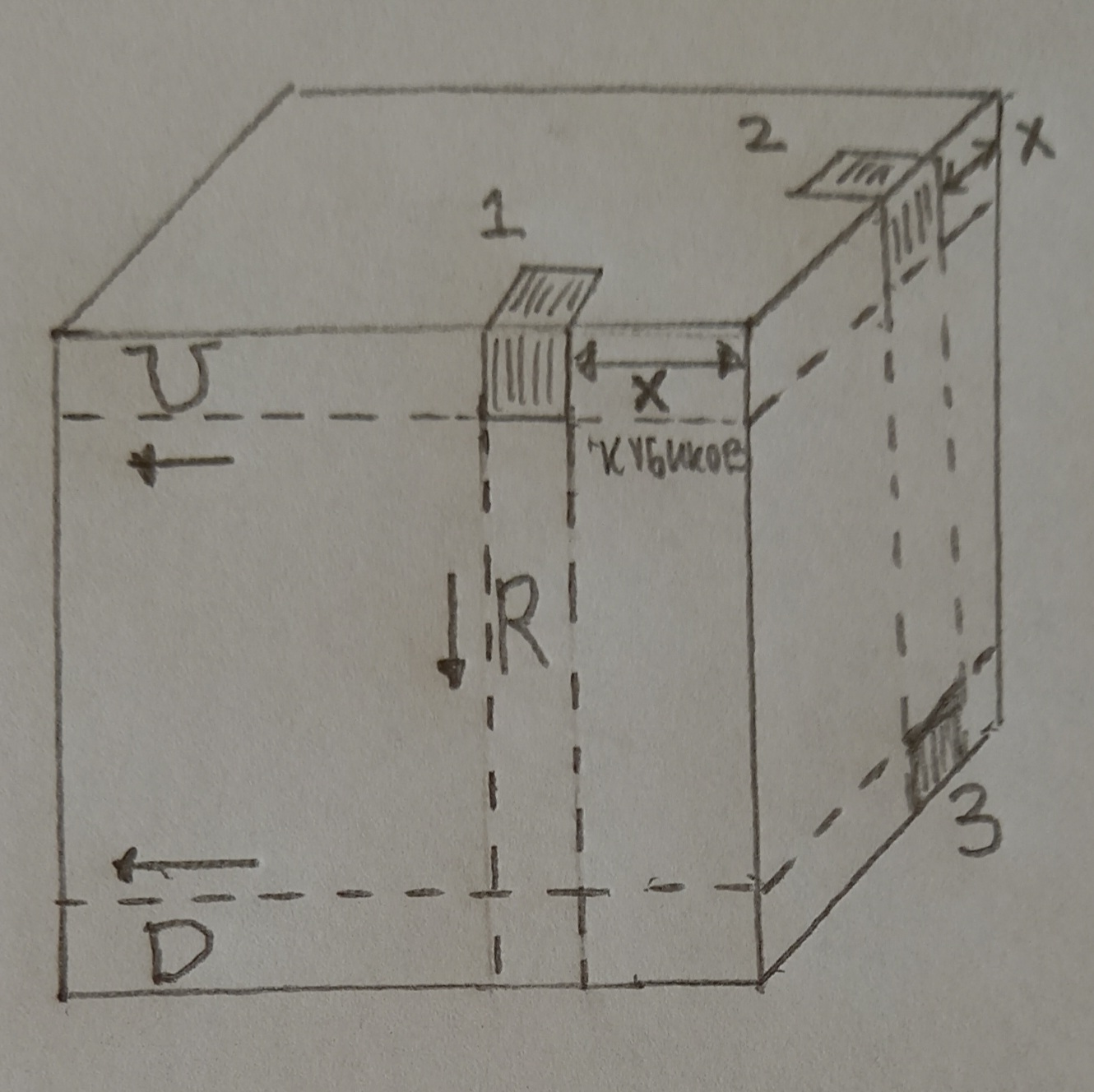}
\caption{}
\label{sec4sub1pic2}
\end{minipage}
\end{center}
\end{figure}

\begin{figure}[h]
\begin{center}
\begin{minipage}[h]{0.4\linewidth}
\includegraphics[width=1\linewidth]{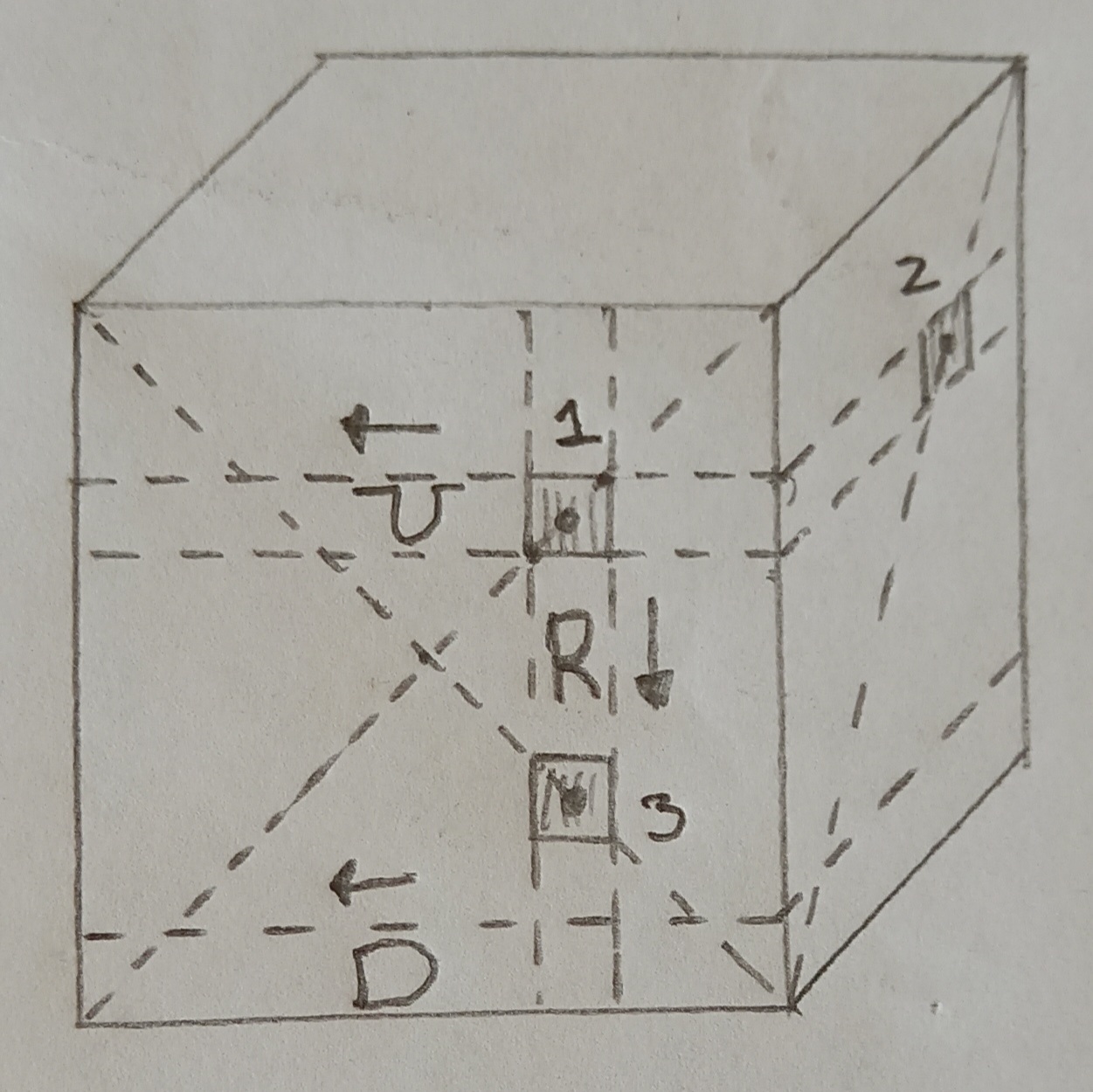}
\caption{}
\label{sec4sub1pic3}
\end{minipage}
\begin{minipage}[h]{0.4\linewidth}
\includegraphics[width=1\linewidth]{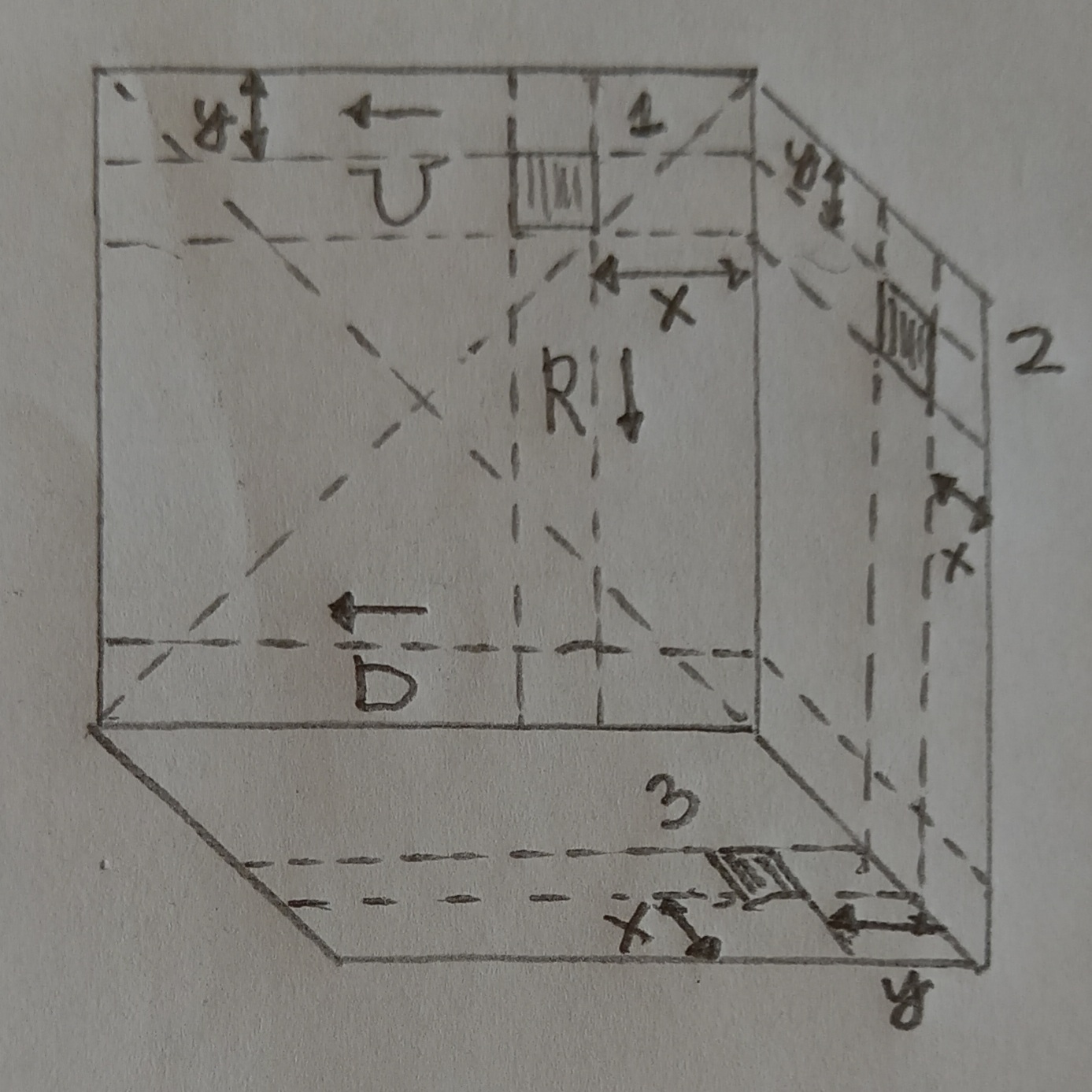}
\caption{}
\label{sec4sub1pic4}
\end{minipage}
\end{center}
\end{figure}

Тогда следующая комбинация
\[
\left[ \left[R,D \right],U \right] = \left[R,D \right] U \left[R,D \right]^{-1} U^{-1}
\]
в каждом из четырёх случаев осуществляет циклическую перестановку, при которой первый кубик переходит на место второго, второй -- на место третьего, а третий -- на место первого.

Вопрос о сведении трёх кубиков на заданные позиции в каждом из случаев также тривиален. Пользуясь принципом аккуратных перестановок, получаем, что любые три кубика можно аккуратно переставить.
\end{proof}

Выберем произвольный $t$-слой. Спроецируем его на $t$-мерное пространство и получим $t$-мерный Кубик Рубика. Для тех кубиков, которые на нём находятся, можно определить какому классу они принадлежат относительно нового Кубика Рубика. Например, угловые кубики в четырёхмерной головоломке принадлежат классу (4), а относительно трёхмерного внешнего слоя, их содержащего, классу (3). Определение классов относительно слоя сыграет роль в следующей лемме.

Укажем, как свести многомерный случай к трёхмерному.

\begin{lemma} \label{3layer}
Любые три кубика одного некаркасного класса можно расположить на некотором трёхмерном слое так, чтобы относительно него они все также принадлежали одному и тому же некаркасному классу.
\end{lemma}

\begin{proof}
\textbf{1.} Случай, когда размерность равна трём, тривиален, так что полагаем $n \geq 4$.

Пусть классом кубиков является $\class$ (или особый $\spec$). Выберем удобный для нас трёхмерный слой в зависимости от канонического положения класса и назовём его также $\textbf{каноническим}$
\[
L_P=
\begin{cases}

\Omega_M^3 \times E_0^{m-3} \times \prod_{i=1}^{n-m} E_{j_i}, & P=E_0^m \times \prod_{i=1}^{n-m} E_{j_i}, \; m\ge 3, \\
\Omega_M^3 \times \prod_{i=2}^{n-m} E_{j_i}, & P=E_0^2 \times \prod_{i=1}^{n-m} E_{j_i}, \; m = 2, \\

\Omega_M^3 \times \prod_{i=3}^{n-m} E_{j_i}, & P=E_V \times \prod_{i=1}^{n-m} E_{j_i}, \; m = 1.

\end{cases}
\]

Обратим внимание, что такой слой содержит $P$, и при этом каноническое положение не является каркасным относительного $L_P$. 

Действительно, воспользуемся упрощённой записью канонического положения и выпишем только первые три множества
\[
\begin{cases}
E_0 \times E_0 \times E_0, & m \ge 3, \\
E_0 \times E_0 \times E_{j_1}, & m = 2, \\
E_V \times E_{j_1} \times E_{j_2}, & m = 1.
\end{cases}
\]
В упрощённой записи каркасного положения содержатся два множества $E_\md$
\[
\prd{E_V E_{\md}^2}.
\]
В первых двух случаях, очевидно, $P$ каркасным быть не может. В третьем это тоже верно. Действительно, для класса с $m=1$, который не является каркасным, минимальная характеристика $j_1 \neq \md$. Так что и тут двух множеств не найдётся.

Отметим, что из упрощённой записи $P$ видно, что при
\begin{itemize}
    \item $m=3$ $P$ является угловым положением относительно слоя,
    \item $m=2$ рёберным положением,
    \item $m=1$ некаркасным двугранным положением.
\end{itemize}
На рисунке \ref{sec4sub1pic5} изображена проекция канонического слоя на трёхмерное пространство и отмечены все три случая $P$.

\begin{figure}
    \centering
    \includegraphics[width=0.6\linewidth]{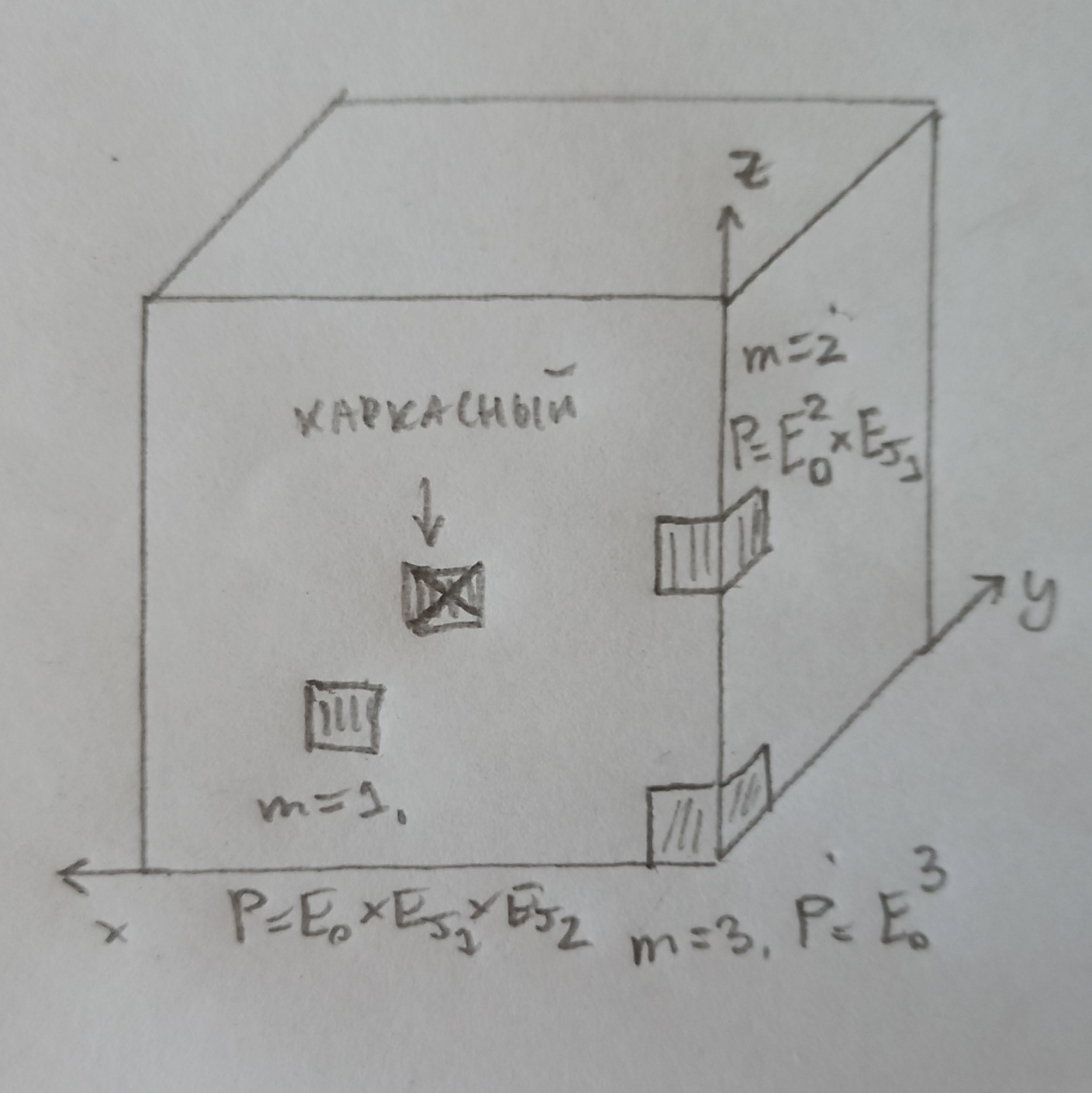}
    \caption{}
    \label{sec4sub1pic5}
\end{figure}

\textbf{2.} Теперь предъявим $\textbf{алгоритм}$ перемещения кубиков на выбранный слой. Занумеруем кубики от одного до трёх.

$\textbf{Шаг 1.}$ Так как каноническое положение связано с положениями кубиков класса, то переведём первый кубик в $P$.

$\textbf{Шаг 2.}$ Рассмотрим некоторую комбинацию поворотов двухмерных слоёв, переводящую второй кубик в $P$. Каждый раз, когда поворот $\psi_{i,j}$ из комбинации затронет положение первого кубика:
\begin{itemize}
    \item Сначала переместим первый кубик в $\textbf{безопасное}$ положение $L_P$, которое не входит в область действия поворота. Перемещение осуществляется $\textbf{в пределах слоя}$, то есть только поворотами слоёв $L_P$ (см. пример для углового $P$ на рис. \ref{sec4sub1pic6}).
    
    В области действия $\psi_{i,j}$ находится кубики одного двумерного слоя Кубика Рубика, в то время как в $L_P$ их как минимум $k \geq 2$. Поэтому всегда удастся выбрать безопасный слой и переместить кубик туда.
    \item Только после этого действия осуществляем $\psi_{i,j}$.
\end{itemize}
В конце данного шага получим, что два кубика располагаются на каноническом слое и связаны с $P$ в пределах $L_P$.

$\textbf{Шаг 3.}$ Аналогично поступаем и с третим кубиком. Единственным отличием будет то, что придётся два кубика одновременно перемещать в безопасные положения (см. пример для углового $P$ на рис. \ref{sec4sub1pic7}).

\begin{figure}[h]
\center
\includegraphics[width=0.6\textwidth]{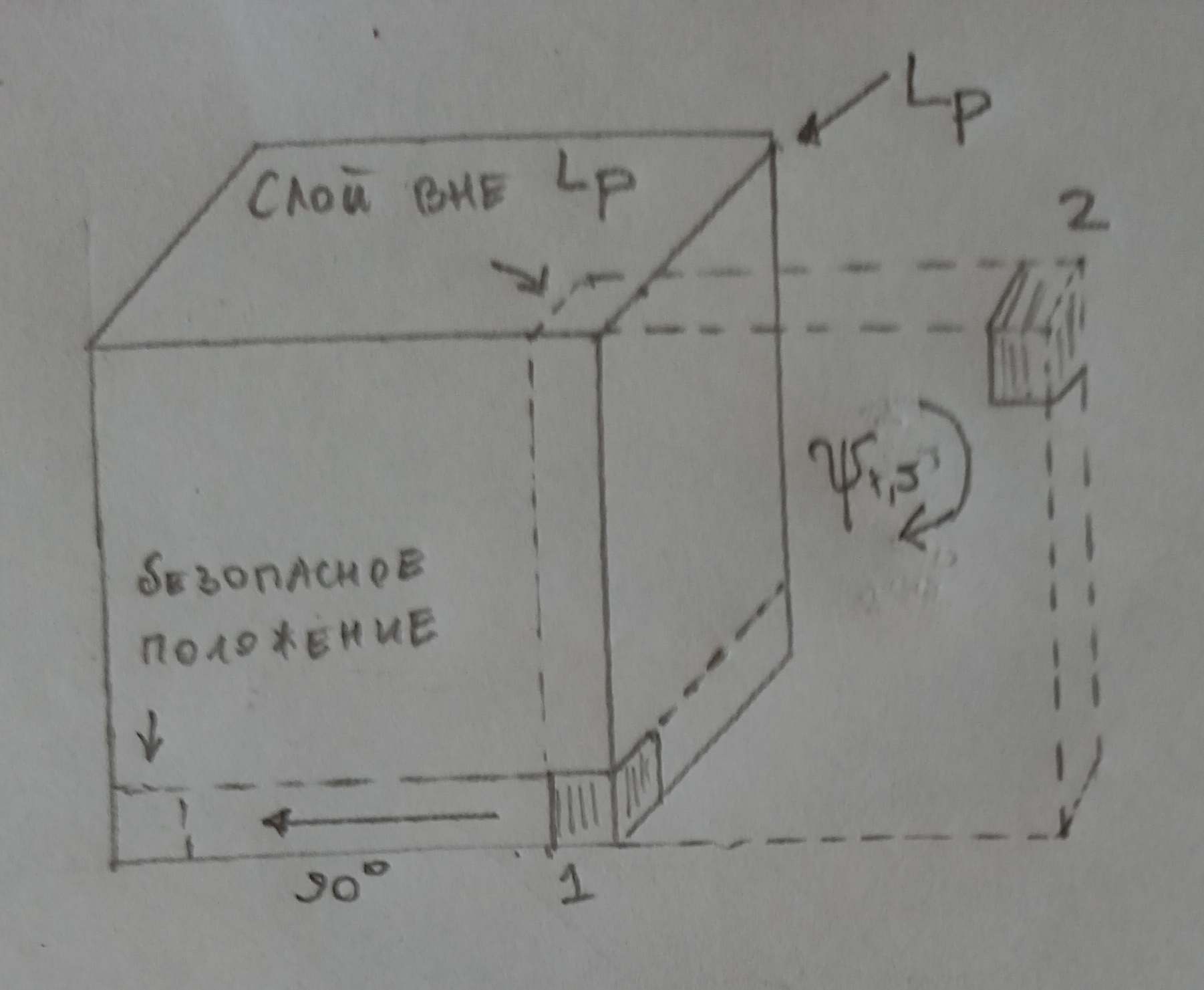}
\caption{}
\label{sec4sub1pic6}
\end{figure}

\begin{figure}[h]
\center
\includegraphics[width=0.6\textwidth]{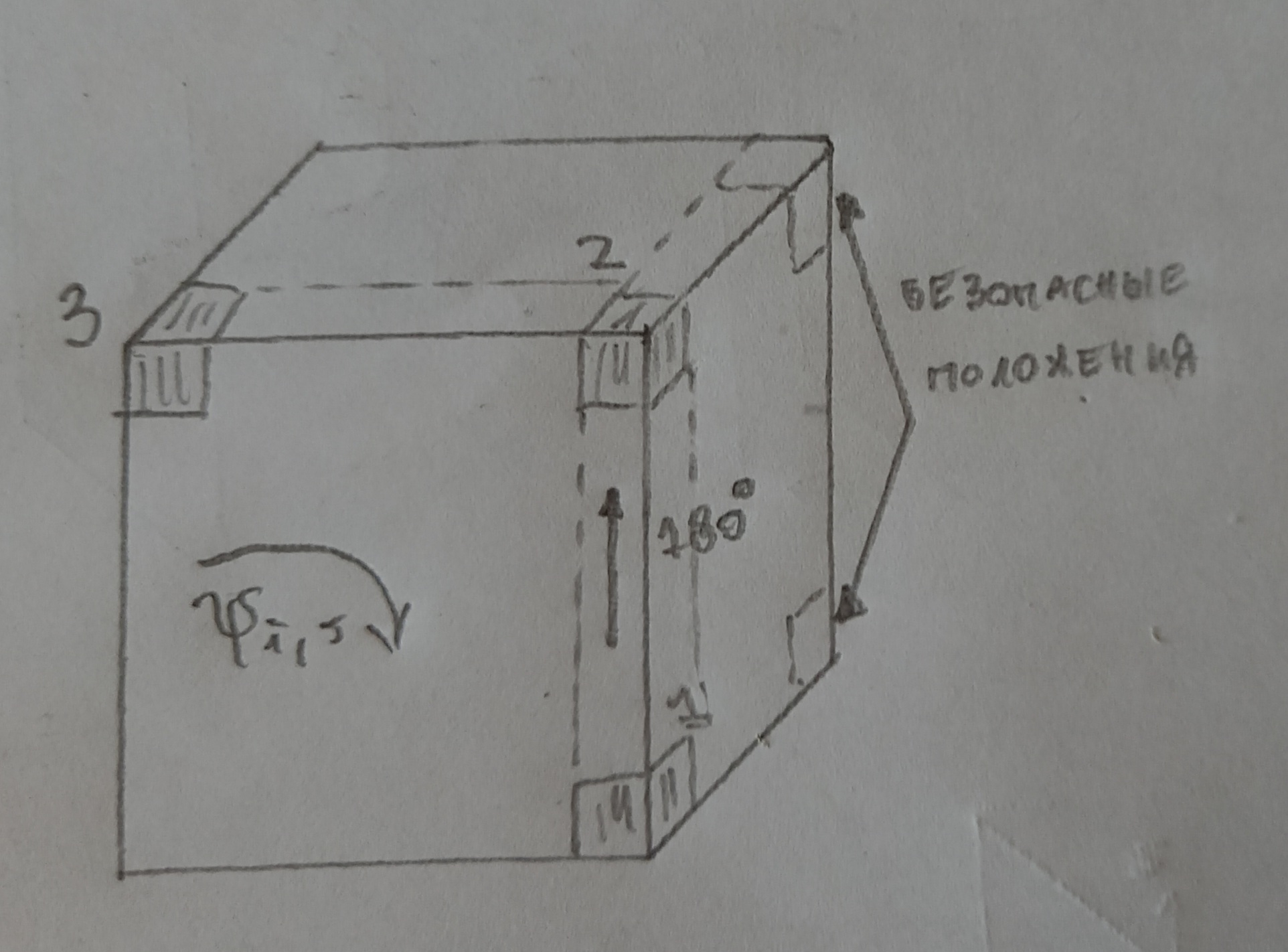}
\caption{}
\label{sec4sub1pic7}
\end{figure}

Выбранные кубики теперь находятся на одном трёхмерном слое и связаны с $P$ в пределах $L_P$ по построению алгоритма. Следовательно, они принадлежат одному классу эквивалентности относительно трёхмерного слоя. В виду некаркасности канонического положения все условия соблюдены и лемма доказана.
\end{proof}

Докажем теорему о перестановке класса.

\begin{proof}
Достаточно доказать, что можно аккуратно циклически переставить три кубика одного некаркасного класса.

Существует $p$, которая переведёт кубики на канонический слой таким образом, чтобы относительно него они принадлежали одному некаркасному классу (см. лемму \ref{3layer}). Существует $q$, аккуратно переставляющая их на каноническом слое (по лемме \ref{3dimpermutation}). Тогда по принципу аккуратной перестановки комбинация $w=pqp^{-1}$ будет искомой циклической перестановкой.
\end{proof}

\subsection{Ориентационные комбинации}

В этом разделе нас не будут интересовать кубики класса (1), поскольку они содержат только одну ориентационную грань.

\begin{lemma} \label{3dimorientation}
Выберем трёхмерный слой $L$ в головоломке, возможно совпадающий с ней целиком, и два его кубика из одного класса $\class$.

Существует аккуратная ориентационная комбинация, переставляющая
\begin{itemize}
    \item три ориентационные грани каждого из двух кубиков, если они угловые относительно $L$.
    \item две ориентационные грани каждого из двух кубиков, если они центрально-рёберные относительно $L$.
    \item две ориентационные грани каждого из двух кубиков, если
    
    1. они нецентрально-рёберные относительно $L$.
    
    2. существуют две совпадающие характеристики их класса $j_p=j_{p+1}$.
\end{itemize}
\end{lemma}

\begin{proof}
Для доказательства леммы понадобится дополнительный третий кубик из $L$ того же класса $\class$. Найдём комбинацию $w$, осуществляющую перестановку граней первого и второго.

Предположим, что удалось найти положения $A,B$ и $C$ из $L$, для которых существует комбинация $s$, удовлетворяющая условиям
\begin{itemize}
    \item Она сохраняет положение $A$.
    \item Она меняет ориентацию кубика из $A$ и не влияет на состояния кубиков из $B$ и $C$.
\end{itemize}

Отметим, что заведомо существует комбинация $p$, перемещающая первый, второй и третий кубики в $A,B$ и $C$. Также по лемме \ref{3dimpermutation} существует $q$, циклически переставляющая кубики из $A$ в $B$, из $B$ в $C$, из $C$ в $A$.

Легко проверить, что $[s,q]$ переставляет грани первого и второго кубиков в положениях $A$ и $B$, не влияя на состояния других. То есть, она является аккуратной ориентационной комбинацией. Тогда по принципу аккуратной перестановки $w=p[s,q]p^{-1}$ будет искомой.

Осталось лишь указать нужные положения и $s$ для каждого класса кубиков. Для этого обратимся к доказательству теоремы о двух группах.

\textbf{1. Пусть класс угловой}. Комбинация
\[
s=\psi_{s,i} \; \circ \; \psi_{i,j}
\]
циклически переставляет грани кубика канонического положения. При этом она задействует вращения двух пересекающихся слоёв. Возьмём положения $A,B$ и $C$ как на рисунке \ref{sec4sub2pic1} и комбинацию вида $s$, но уже применённую к положению $A$ слоя $L$. Тогда все условия на $s$ будут выполнены.

\begin{figure}[h]
    \centering
    \includegraphics[width=0.6\linewidth]{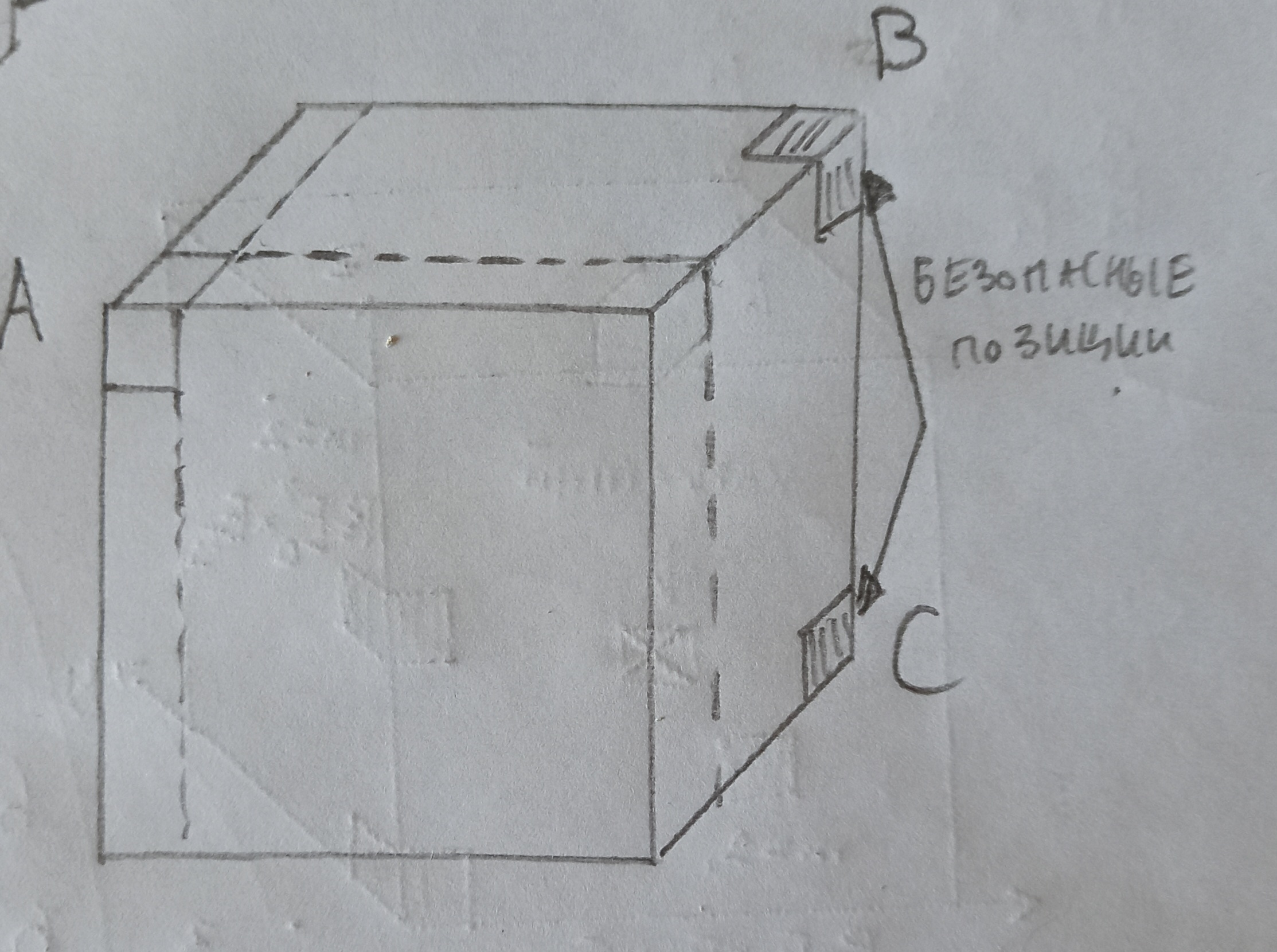}
    \caption{}
    \label{sec4sub2pic1}
\end{figure}

\textbf{2. Пусть класс центрально-рёберный}. Комбинация
\[
s=\psi^2_{i,m} \circ \psi_{i,j}
\]
совершает транспозицию граней кубика в положении $P$. Дальнешие рассуждения, как и в случае угловых. Соответствующий пример положений $A,B$ и $C$ приведён на рисунке \ref{sec4sub2pic2}.

\begin{figure}[h]
    \centering
    \includegraphics[width=0.6\linewidth]{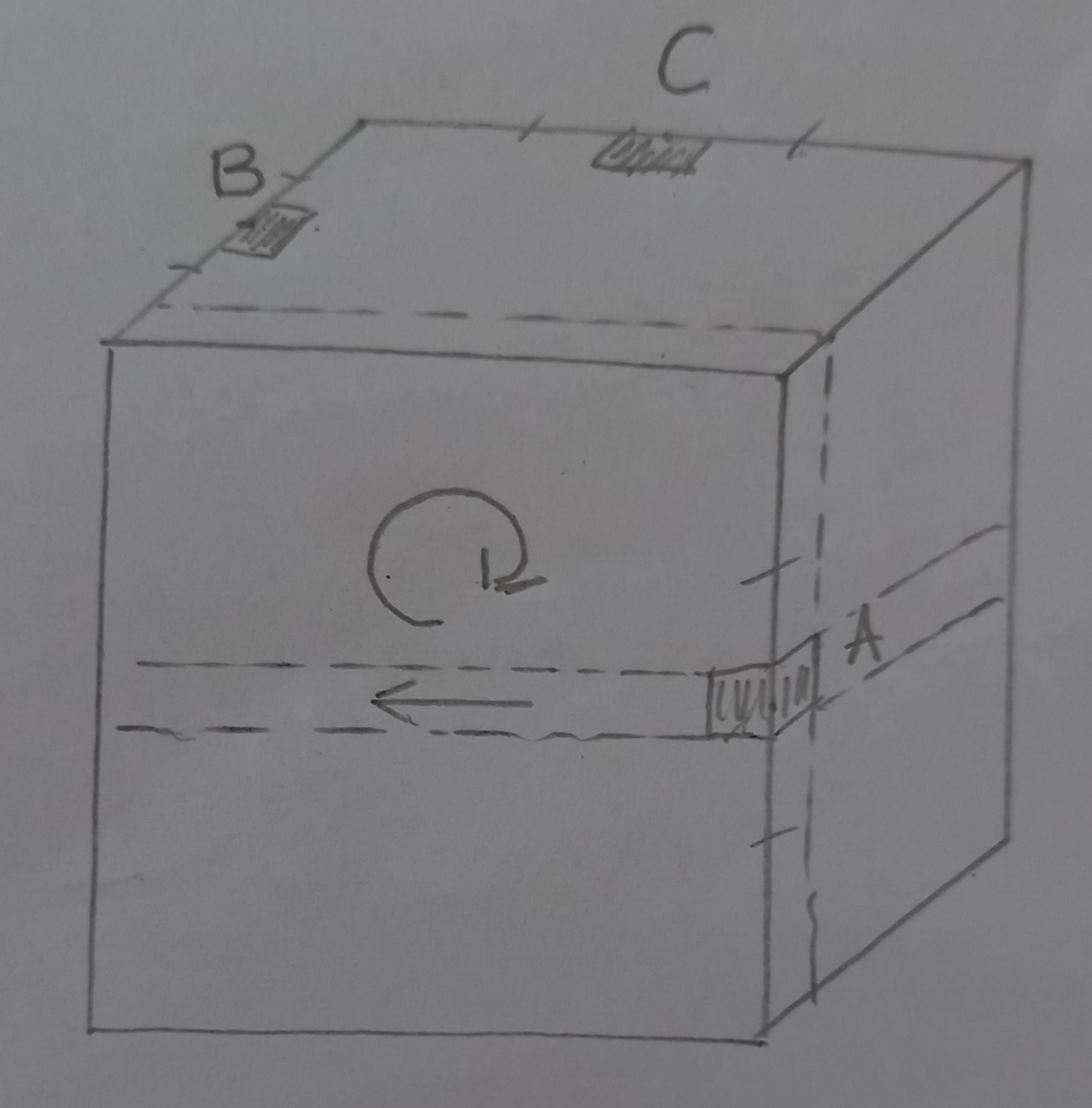}
    \caption{}
    \label{sec4sub2pic2}
\end{figure}

\textbf{3. Пусть класс нецентрально-рёберный}. Комбинация
\[
s=\psi_{m+p+1,m+p} \circ \psi^2_{i,m+p} \circ \psi_{i,j}
\]
совершает транспозицию граней кубика $P$ в головоломке размерности большей трёх, так как существуют две совпадающие характеристики. В размерности три такое невозможно из-за того, что зависимая группа состоит из тождественной перестановки. Применим подобную комбинацию $s$ к положению $A$ слоя $L$ (рис. \ref{sec4sub2pic3}). При таких поворотах задействуются вращения двух пересекающихся слоёв, которые переводят кубик из $A$ в $A^{'}$, а затем слой вне $L$ переводит его обратно в $A$. Этот внешний слой пересекается с $L$ только по 'одномерному слою', содержащему $A$. Потому легко выбрать безопасные позиции $B,C$.

\begin{figure}[h]
    \centering
    \includegraphics[width=0.6\linewidth]{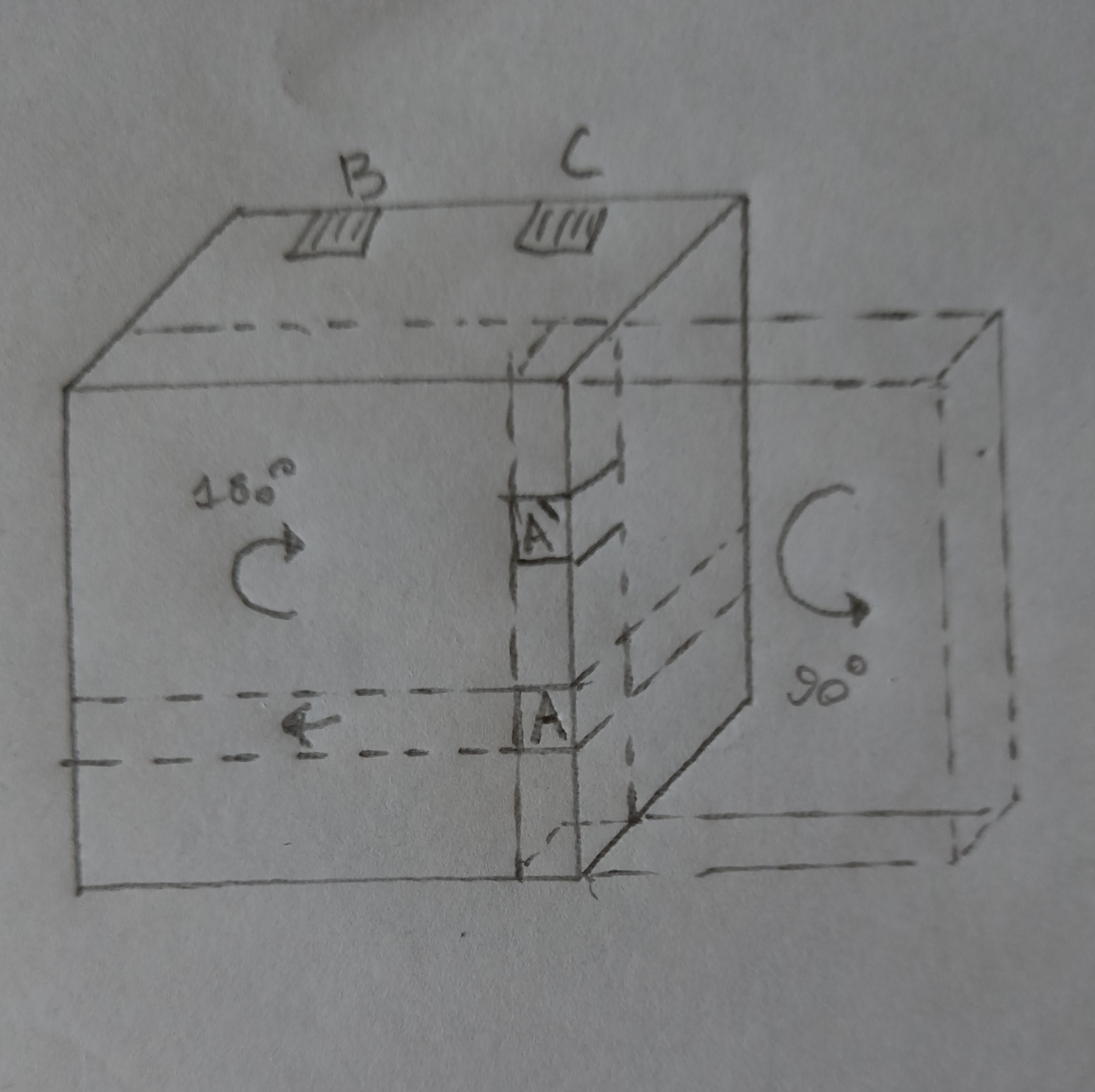}
    \caption{}
    \label{sec4sub2pic3}
\end{figure}

\end{proof}

\begin{theorem} \label{twocubesorient}
Для любой пары кубиков из одного класса существует аккуратная ориентационная комбинация. При этом у одного из кубиков ориентацию можно изменить любым возможным образом. То есть, для произвольной ориентационной перестановки $z\in \oG$ существует комбинация поворотов, после которой для одного из кубиков
\[
\phi^{'}=z \cdot \phi.
\]
\end{theorem}

\begin{proof}
Пусть их класс -- $\class$. Для определённости положим, что таблица ориентационных граней выглядит как
\begin{center}
\begin{tabular}{c c c c c c c c c}
 1 & $E_V$ & $V$ & $V$ & \dots & $V$ & $P_{j_1}$ & \dots & $P_{j_{n-m}}$ \\
 2 & $V$ & $E_V$ & $V$ & \dots & $V$ & $P_{j_1}$ & \dots & $P_{j_{n-m}}$ \\
 3 & $V$ & $V$ & $E_V$ & \dots & $V$ & $P_{j_1}$ & \dots & $P_{j_{n-m}}$ \\
 \dots & \dots & \dots & \dots & \dots & \dots & \dots & \dots & \dots \\ 
 $m$ & $V$ & $V$ & $V$ & \dots & $E_V$ & $P_{j_1}$ & \dots & $P_{j_{n-m}}$.
\end{tabular}
\end{center}
Необходимо доказать теорему в случае, когда $z$ -- перестановка двух граней (если $\oG=S_m$) или трёх (если $\oG=A_m$). В первом случае выберем первые две грани, во втором -- первые три. В случае выбора каких-либо других граней доказательство будем аналогичным.

Для начала в каждом из трёх случаев (в зависимости от $m$ и характеристик) укажем, на какой трёхмерный слой $L$ мы переведём оба кубика.

\begin{itemize}
    \item Если $m=n$ или $j_1 < \dots < j_{n-m} < \mdd$, то $\oG=A_m$ и
    \[
    L=\Omega_M^3 \times E_0^{m-3} \times \prod_{i=1}^{n-m} E_{j_i}.
    \]
    Относительно него кубики будут угловыми.
    \item Если $m<n$ и $j_{n-m}=\frac{k-1}{2}$, то $\oG=S_m$ и
    \[
    L=\Omega_M^2 \times E_0^{m-2} \times \prod_{i=1}^{n-m} E_{j_i} \times \Omega_M.
    \]
    Относительно него кубики центрально-рёберные.
    \item Если $m<n$ и $\exists l: j_l=j_{l+1}<\mdd$, то $\oG=S_m$ и
    \[
    L=\Omega_M^2 \times \prod_{i=1}^{l-1} E_{j_i} \times \Omega_M \times \prod_{i=l+1}^{n-m} E_{j_i}
    \]
    Относительно него кубики нецентрально-симметричные с совпадающими двумя характеристиками.
\end{itemize}
Легко проверить, что все выбранные грани содержатся в соответствующем слое $L$. Далее используя лемму \ref{3dimorientation}, доказываем утверждение теоремы.
\end{proof}

\begin{remark}
Данная теорема позволяет менять ориентацию класса 'почти' произвольным образом. Например, занумеруем кубики класса от 1 до $N$, а затем будем менять ориентации 1-го и 2-го, 1-го и 3-го,$\dots$,1-го и $N$-го. Тогда мы можем добиться любых ориентаций у кубиков 2,3,$\dots , N$ по теореме \ref{twocubesorient}. Останется неориентированным лишь первый их них. Возникает вопрос: какие его ориентации осуществимы, если все остальные кубики класса должны остаться в тех же состояниях?

Теорема \ref{onecubeorient} даёт ответ на вопрос, какие ориентации заведомо осуществимы, но не раскрывает детали об остальных. В дальнейшем станет ясно, что никакие другие осуществить и не удастся из-за сохранения ориентационного инварианта.
\end{remark}

\begin{deff}
\textbf{Коммутатором} комбинациq $a$ и $b$ называется следующая комбинация
\[
[a,b]=a b a^{-1} b^{-1}.
\]

\textbf{Коммутантом} группы $G$ является группа, порождённая всеми коммутаторами $G$
\[
Q=[G,G]= \: <[g_1,g_2] \; | \; g_1, g_2 \in G>.
\]
\end{deff}

\begin{theorem} \label{onecubeorient} 
Для любого кубика с зависимой группой $\oG$ выберем произвольную ориентационную подстановку $z$ из коммутанта $\left[\oG, \oG \right]$. Тогда для кубика найдётся аккуратная комбинация, меняющая его ориентацию следующим образом:
\[
\phi^{'}=z \cdot \phi.
\]
\end{theorem}

\begin{proof}
Так как произвольная подстановка из коммутанта является произведением коммутаторов, то достаточно доказать только для случая одного коммутатора.

Пусть $z=[z_1,z_2]$ для некоторых $z_1$ и $z_2$ из $\oG$. Выберем ещё два любых кубика того же класса, что и данный, и пронумеруем их. Исходный кубик будет первым. По теореме \ref{twocubesorient} существует

\begin{itemize}
    \item комбинация $x$, меняющая ориентации первого и второго кубика, при этом:
    \[
    \phi_1^{'}=z_1 \cdot \phi_1
    \]
    \item комбинация $y$, меняющая ориентации первого и третьего кубика, при этом:
    \[
    \phi_1^{'}=z_2 \cdot \phi_1
    \]
\end{itemize}

Тогда $\left[y^{-1},x^{-1}\right]=y^{-1}x^{-1}yx$ меняет ориентацию только первого кубика:
\[
\phi^{'}=z_1 \cdot z_2 \cdot z_1^{-1} \cdot z_2^{-1} \cdot \phi=[z_1,z_2] \cdot \phi = z \cdot \phi.
\]
\end{proof}

\section{Инварианты}

\subsection{Каркасные инварианты}

Каркасные кубики - кубики класса $(1,\md,\dots,\md)$. Их положения записываются как:
\[
\prd{E_V E^{n-1}_{\md}}.
\]
Несложно убедиться, что всего таких кубиков $2n$. Каждый имеет одну ориентационную грань и является уникальным, что следует из теоремы о кластерах.

\begin{theorem}
Выберем произвольное состояние каркаса $S$ и определим для него две величины.

1. Разобьём все кубики на $n$ пар с положениями:
\[
\prd{E_0 E^{n-1}_{\md}} \text{ и } \prd{E_{k-1} E^{n-1}_{\md}}.
\]

Состав пар при поворотах не меняется (в паре всегда одни и те же элементы). Таким образом, если осуществить разбиение $2n$ уникальных кубиков на $n$ пар, то это разбиение инвариантно относительно вращения Кубика Рубика.

Пронумеруем все такие разбиения на пары и определим \textbf{первый каркасный инвариант} $B_1$, как номер разбиения для $S$
\[
B_1 \in \Z_N, \; N=\frac{(2n)!}{2^n \cdot n!}.
\]

2. Выберем фиксированное состояние каркаса $S^{'}$. Обозначим за $PS$ и $PS^{'}$ состояния множеств пар, которые соответствуют $S$ и $S^{'}$.

Пусть $S$ получается из $S^{'}$ перестановкой $\sigma_c$ его кубиков, а $PS$ из $PS^{'}$ перестановкой $\sigma_p$ его пар. Тогда величина
\[
B_2=sgn (\sigma_c) \cdot sgn (\sigma_p) \in \Z_2
\]
является \textbf{вторым каркасным инвариантом}.
\end{theorem}

\begin{proof}
\textbf{1.},Для определённости выберем некоторый слой, содержащий кубик с положением
\[
E_0 \times E^{n-1}_{\md},
\]
и осуществим поворот. Понятно, что если поворот будет $\psi_{i,j}$, где $i,j > 1$, то он не изменит положения кубика.

Опять же, для определённости, возьмём $\psi_{1,2}$. Тогда произойдёт циклическая перестановка четвёрки кубиков (далее -- упрощённая запись)
\[
E_0 \times E_\md, \; E_\md \times E_0, \; E_{k-1} \times E_\md, \; E_\md \times E_{k-1},
\]
откуда видно, что произошла перестановка двух пар. Значит если выбрать какую-либо пару, то кубики в ней всегда будут одни и те же, так как пары переставляются целиком.

Разбиение инвариантно, а количество способов разбить $2n$ объекта на $n$ пар равно
\[
N=\frac{(2n)!}{2^n \cdot n!}.
\]

2. Пусть совершается поворот, изменяющий положение каркаса. Он переставляет четыре кубика и, следовательно, 2 пары. Тогда обе чётности $sgn (\sigma_c)$ и $sgn (\sigma_p)$ меняются одновременно, откуда следует инвариантность $B_2$.
\end{proof}

\subsection{Центральный инвариант}
    
\begin{theorem}
Пусть $k$ нечётно. Выберем произвольное фиксированное состояние $S^{'}$ и определим инвариант для произвольного состояния куба $S$.

Пусть положения центрального класса
\[
(m,\md,\dots,\md)
\]
в $S$ и $S^{'}$ отличаются перестановкой $\sigma_m$ его элементов.

Тогда следующая величина
\[
C=\prod\limits_{m=1}^n sgn (\sigma_m) \in \Z_2
\]
является инвариантом, называемым \textbf{центральным}.
\end{theorem}
    
\begin{proof}
Покажем, что $C$ инвариантно относительно поворота произвольного слоя. Воспользуемся следующими приёмами для упрощения задачи
    
\begin{itemize}
    \item Добавим в произведение перестановку для центрального класса при $m=0$. Он состоит из одного элемента, центрального внутреннего кубика, и следовательно $\sigma_0$ всегда тождественная. Тогда на величину $C$ это не повлияет.
    \item Существует изоморфизм между множеством центральных классов с операцией поворота слоя и $n$-мерным Кубиком Рубика с ребром 3 (рис. \ref{sec5sub1pic1}):
    \[
    \prd{E_{V_1} \dots E_{V_m} E^{n-m}_\md} \Longleftrightarrow \prd{E_{V_1} \dots E_{V_m} E^{n-m}_1}.
    \]
    Он осуществляет замену $k$ на 3 в записи слоёв и положений кубиков и их граней.
\end{itemize}
    
\begin{figure}
    \centering
    \includegraphics[width=0.8\linewidth]{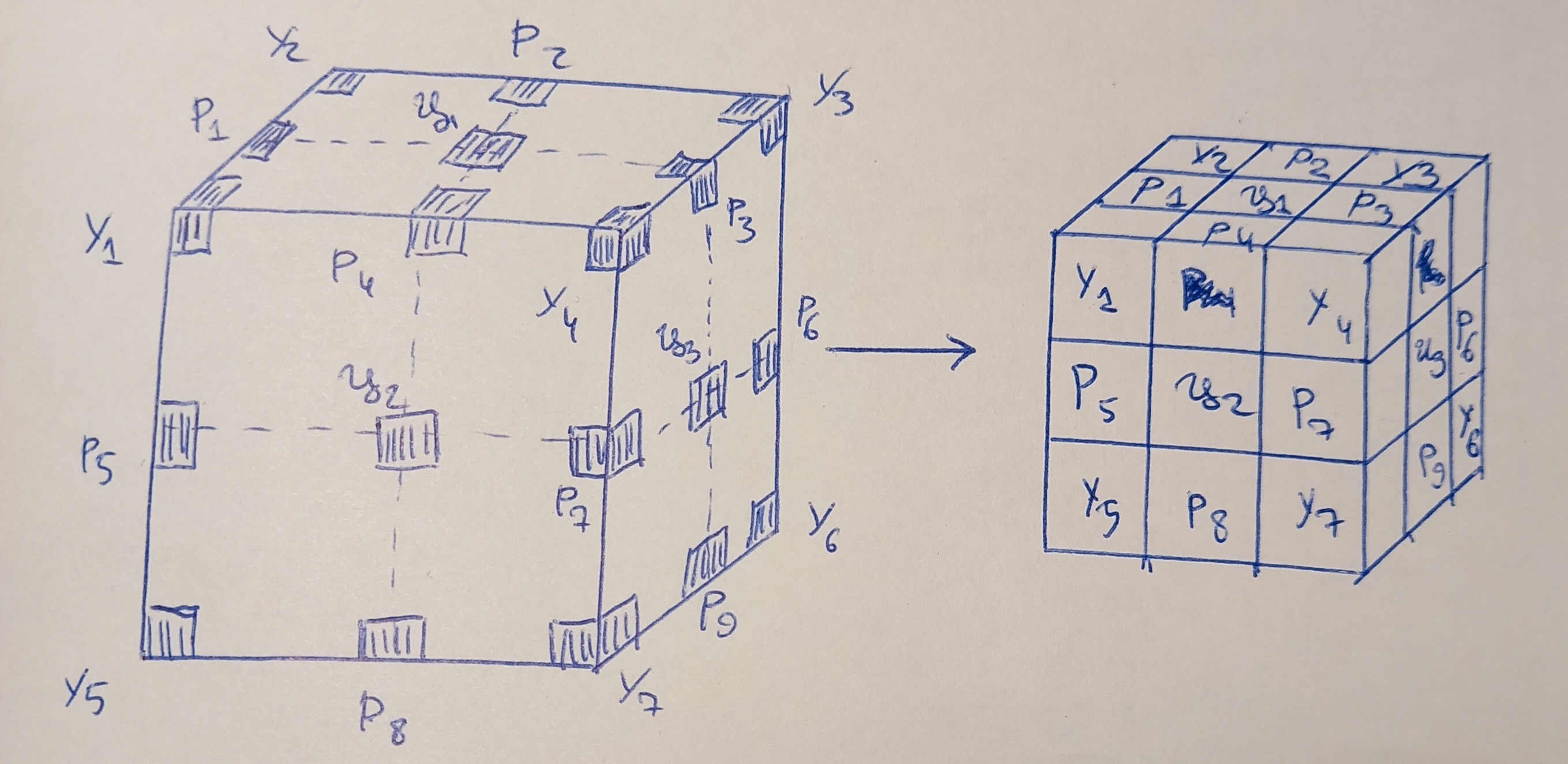}
    \caption{}
    \label{sec5sub1pic1}
\end{figure}

$\psi_{ij}$ $t$-мерного слоя не переставляет кубики, у кототорых $i$-ое и $j$-ое множества положений равны $E_1$. Остальные же множества положений таких кубиков могут быть произвольными. Тогда из $3^t$ элементов слоя $3^{t-2}$ инварианты относительно поворота. Все остальные разбиваются на циклически переставляющиеся четвёрки. Каждая из них осуществляет нечётную перестановку, но так как всего четвёрок чётное число
\[
\frac{3^t - 3^{t-2}}{4} = 2 \cdot 3^{t-2}, 
\]
то $C$ сохраняется.
\end{proof}
    
При доказательстве полноты системы инвариантов нам понадобится следующая лемма.
    
\begin{lemma} \label{central}
Поворот $t$-мерного внешнего слоя не меняет положения класса
\[
(m,\md, \dots ,\md) \text{ при } m \leq (n-t),
\]
но осуществляет нечётную перестановку класса
\[
(n-t+1,\md, \dots ,\md).
\]
\end{lemma}
    
\begin{proof}
Для определённости возьмем слой
\[
\Omega_M^t \times E_0^{n-t}
\]
и его поворот $\psi_{1,2}$.

\textbf{1.} Пусть $m \le (n-t)$. Согласно определению запись положения кубика из $(m,\md,\dots,\md)$ содержит $m$ множеств $E_V$, в то время как для элементов слоя их должно быть не менее $(n-t)$. Следовательно, при $m<(n-t)$ элементы класса не содержатся в слое.

Если же $m=(n-t)$, то положение кубика класса записывается единственым образом
\[
E_{\md}^t \times E_0^{n-t}.
\]
Так как первые два множества -- $E_\md$, то положение инвариантно относительно $\psi_{1,2}$.

\textbf{2.} При $m=(n-t+1)$ среди первых $t$ множеств положения будет ровно одно $E_V$, а остальные -- $E_\md$. Так как нас прежде всего интересуют кубики, на положения которых $\psi_{1,2}$ влияет, то первыми двумя множествами положения не могут быть $E_\md$. Легко увидеть, что лишь четыре кубика меняют позиции при повороте (далее - упрощённая запись)
\[
E_0 \times E_{\md}, \; E_{\md} \times E_0, \; E_{k-1} \times E_{\md}, \; E_{\md} \times E_{k-1},
\]
и происходит это циклически. Значит, осуществляется нечётная перестановка класса.
\end{proof}

\subsection{Ориентационный инвариант}

\begin{theorem}
Пусть дан класс кубиков $\class$, для которого
\[
Q=\left[\oG,\oG\right] \neq \oG.
\]
Обозначим за $\phi_i$ ориентационную перестановку его $i$-го элемента, а за $N$ их общее количество.

Тогда смежный класс
\[
O=\left(\prod\limits_{i=1}^{N} \phi_i\right) Q \in \oG/Q,
\]
является \textbf{ориентационным} инвариантом.
\end{theorem}

\begin{proof}
Факторгруппа по коммутанту является коммутативной. Это вытекает из того, что:
\[
\forall a,b \in \overline G_{rot}: ab(ba)^{-1}=aba^{-1}b^{-1}=[a,b] \in Q \Longrightarrow ab \sim ba.
\]
Обозначим за $\phi_i^{'}$ ориентационную перестановку $i$-го кубика после поворота. Далее под перестановками $\phi$ и $\alpha$ имеем в виду соответствующие смежные классы $\phi \, Q$ и $\alpha \, Q$.

Произведение по ориентационным перестановкам разбивается на произведение по четвёркам кубиков и по кубикам, не меняющим положения. Для последних, конечно, $\phi^{'}=\phi$. Для первых, в силу коммутативности факторгруппы и леммы об ориентации, имеем
\[
\phi_{1}^{'} \phi_{2}^{'} \phi_{3}^{'} \phi_{4}^{'} = \phi_{1} \alpha_{1} \phi_{2} \alpha_{2} \phi_{3} \alpha_{3} \phi_{4} \alpha_{4} = (\phi_{1} \phi_{2} \phi_{3} \phi_{4}) \cdot (\alpha_{4} \alpha_{3} \alpha_{2} \alpha_{1}) = \phi_{1} \phi_{2} \phi_{3} \phi_{4}.
\]
Таким образом, $O$ -- инвариант.
\end{proof}

\begin{Classorient}
Предположим, что даны два состояния класса, которые отличаются лишь ориентациями кубиков. Если
\begin{itemize}
    \item их ориентационные инварианты совпадают (или отсутствуют)
    \item для каждого кубика справедливо, что ориентационная перестановка $\phi_i^{'}$ в одном состоянии достижима из $\phi_i$ в другом. Иначе говоря,
    \[
    \phi_i^{'} = \phi_i \cdot z
    \]  
    для некоторой $z$ из зависимой группы
\end{itemize}
то существует аккуратная комбинация, ориентирующая класс, переводя его из одного состояния в другое.
\end{Classorient}

\begin{proof}
Если ориентационный инвариант отсутствует, то зависимая группа совпадает со своим коммутантом. Тогда по теореме \ref{onecubeorient} можно аккуратно изменить ориентацию каждого кубика произвольным образом.

Пусть оба инварианты существуют. Ориентируем все элементы класса кроме одного, как это было сделано в замечании к теореме \ref{twocubesorient}.

Ориентационный инвариант в обоих состояниях один и тот же, как и ориентации всех кубиков, кроме одного. Это означает, что ориентационные перестановки оставшегося кубика в этих состояниях принадлежат одному и тому же смежному классу. То есть,
\[
\phi^{'} = \phi \cdot z
\]
для некоторой $z$ из коммутанта, а такое ориентирование заведомо осуществимо по теореме \ref{onecubeorient}.
\end{proof}

\subsection{Кластерный инвариант}

\begin{deff}
Пусть дан кластер знакопеременного класса $\class$, где $m>1$. Тогда, он раскрашен в $m$ цветов. Занумеруем ориентационные грани кубиков кластера согласно раскраске: каждому цвету свое число от 1 до $m$. Также сопоставим каноническому положению правильную ориентацию.

\textbf{Упрощённой ориентацией} кубика данного кластера, называется число $or$, равное
\begin{itemize}
    \item 1, если можно перевести кубик в положение $P$ и правильно ориентировать.
    \item 0, если нельзя перевести кубик в $P$ и одновременно с этим правильно ориентировать.
\end{itemize}
\end{deff}

Кластеры с упрощённой ориентацией указаны на рисунках \ref{sec5sub4pic1} и \ref{sec5sub4pic2}. На первом из них пример упрощённых ориентаций трёх нетривиальных кластеров трёхмерного Кубика Рубика. На втором рисунке указан четырёхмерный кластер из 8 элементов, каждый из которых имеет 2 трёхмерные ориентационные грани.

\begin{remark}
Исходное определение можно сформулировать немного иначе.

Определим правильную ориентацию для $P$. Затем кубик канонического положения и правильной ориентации переместим в произвольное положение $A$. Его новая ориентация считается правильной ориентацией $A$.

Такая нумерация граней и выбор правильных ориентаций приводит к тому, что новые ориентационные перестановки кубиков кластера уже не обязаны лежать в зависимой группе.
\end{remark}

Так как класс знакопеременный (и не угловой), то $\oG=A_m$ и $\Gr=S_m$. Следовательно факторгруппа $\Gr/\oG \cong \Z_2$. Значит, все ориентационные перестановки принадлежат одному из двух классов смежности по зависимой группе. Если две перестановки лежат в одном классе, то можно перейти от одной ориентации кубика к другой.

Теперь легко заметить, что упрощённая ориентация лишь указывает на класс смежности, которому принадлежит ориентационная перестановка кубика. Если в двух состояниях кубика одна и та же упрощённая ориентация, то можно от обычной ориентации кубика в одном состояний перейти к ориентации во втором.

\begin{figure}[h]
\begin{center}
\begin{minipage}[h]{0.44\linewidth}
\includegraphics[width=1\linewidth]{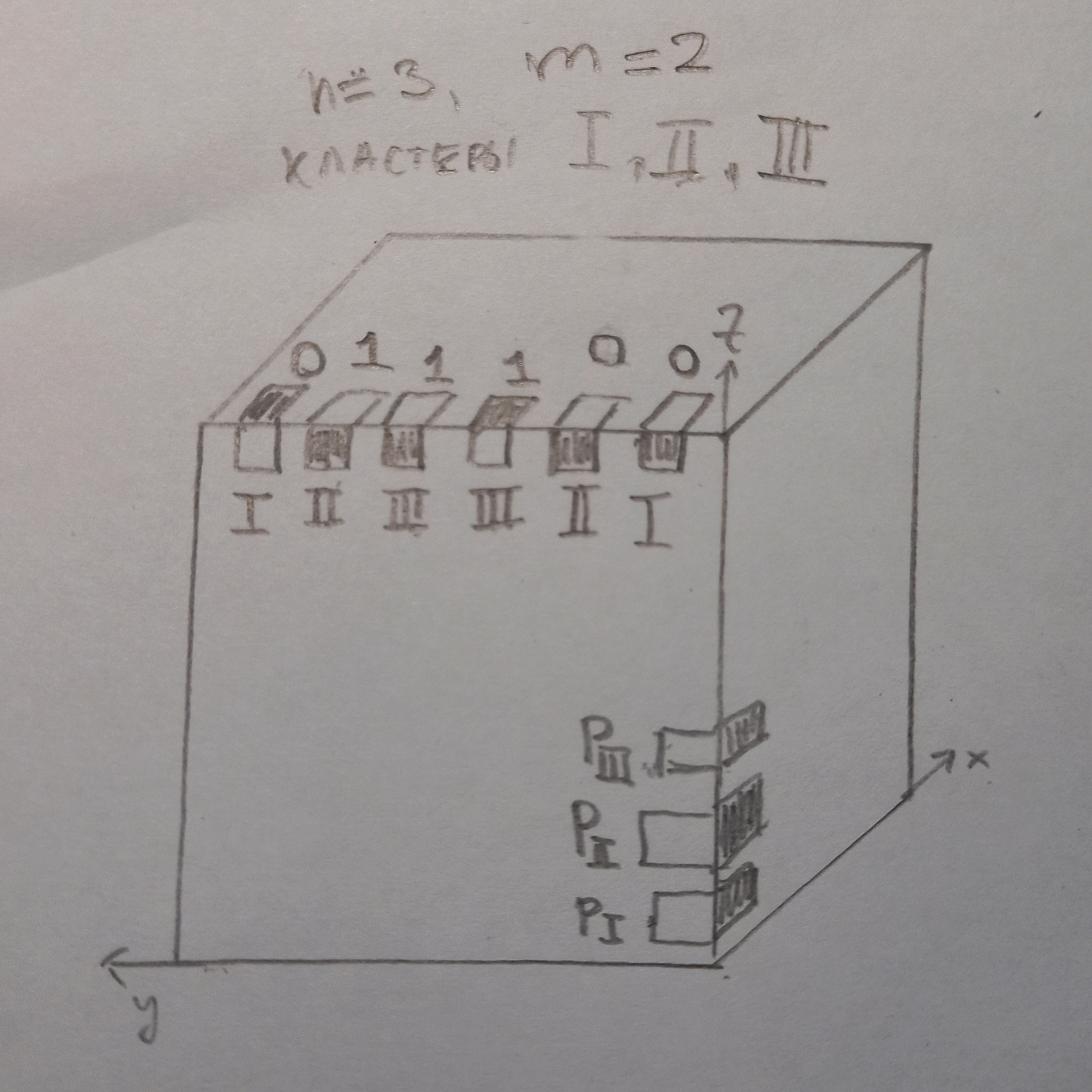}
\caption{}
\label{sec5sub4pic1}
\end{minipage}
\begin{minipage}[h]{0.55\linewidth}
\includegraphics[width=1\linewidth]{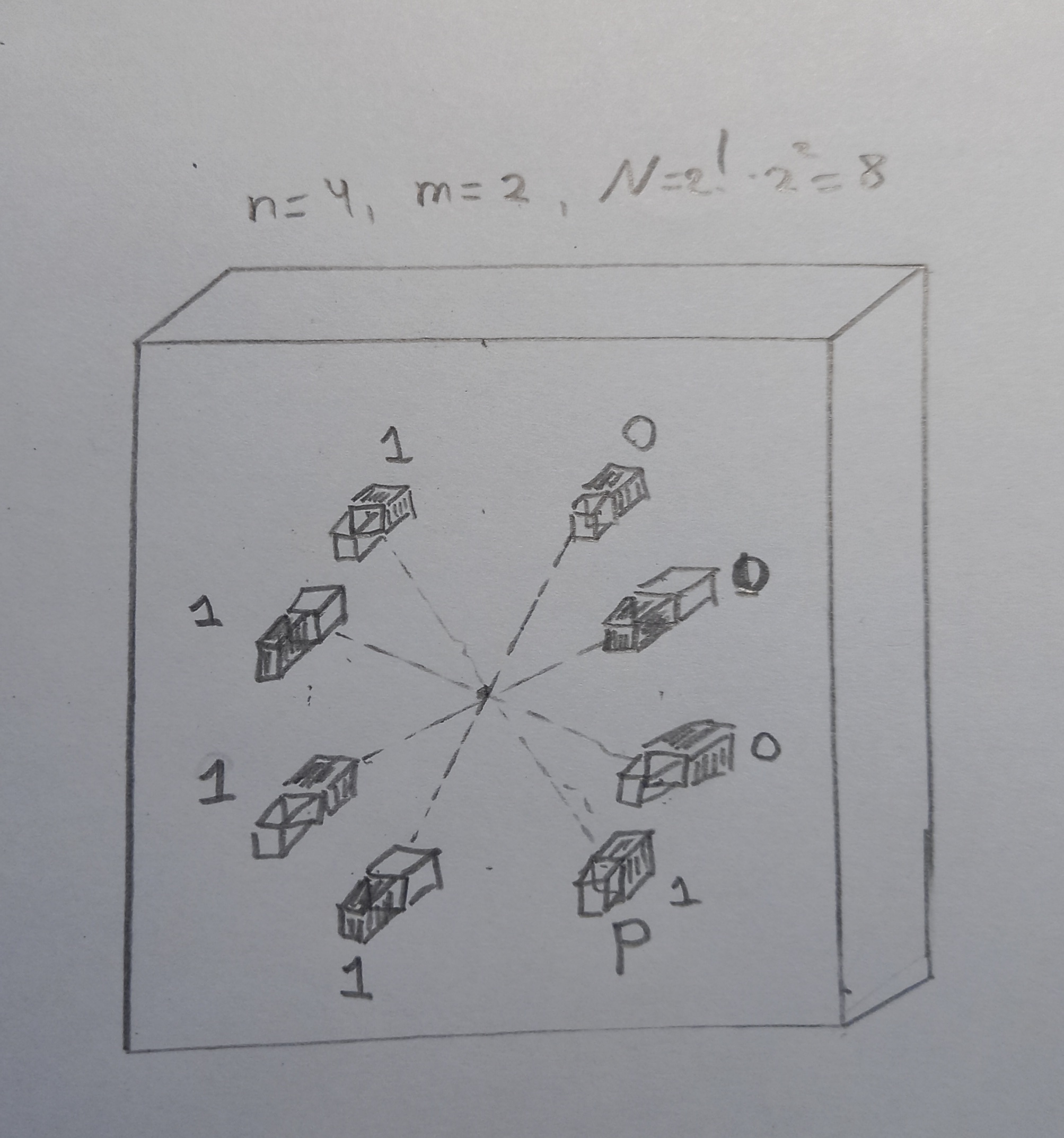}
\caption{}
\label{sec5sub4pic2}
\end{minipage}
\end{center}
\end{figure}

\begin{theorem}
Введём упрощённую ориентацию для знакопеременного класстера из $\class$, где $m>1$.

За $or_i$ обозначим упрощённую ориентацию $i$-го кубика, за
\[
N=(n-m)! \cdot 2^{n-m}
\]
общее количество кубиков кластера. Тогда величина
\[
Cl=\sum_{i=1}^{N} or_i \in \Z_{N+1}
\]
является инвариантом. Будем называть его \textbf{кластерным}.
\end{theorem}

\begin{proof}
$Cl$ равен количеству кубиков, которые можно перевести в $P$ и правильно ориентировать. Предположим, что у некоторых двух связанных состояний кластера инварианты различны. Это означает, что в кластере найдётся кубик, который одновременно можно и нельзя правильно ориентировать в положении $P$. Пришли к противоречию.
\end{proof}

\section{Полная система инвариантов}

\subsection{Описание полной системы инвариантов}

Ниже представлена система всех описанных ранее инвариантов. Её полнота будет доказана в следующем разделе.

\begin{itemize}
  \item $B_1 \in \Z_{\frac{(2n)!}{2^n \cdot n!}}$ и $B_2 \in \Z_2$ - каркасные инварианты, существующие для нечётного $k$.
  \item $C \in \Z_2$ - центральный инвариант, существующий также только для нечётного $k$.
  \item $O \in \oG/Q$ - ориентационный инвариант, определённый для каждого класса, у которого $Q \neq \oG$.
  \item $Cl \in \Z_{(n-m)! \cdot 2^{n-m}+1}$ - кластерный инвариант, определённый для любого нетривиального знакопеременного кластера из
  \[
  \class.
  \]
\end{itemize}

Вспомним, что $\binom{n}{s}$ обращается в нуль при $n\leq 0$, или $s<0$, или $s>n$.

\begin{theorem}
Количество попарно несвязанных состояний $n$-мерного Кубика Рубика с ребром длины $k$ равно
\[
S(n,k)=BC(n,k) \cdot OA(n,k) \cdot OS(n,k) \cdot Cl(n,k),
\]
\[
BC(n,k)=\begin{cases}
1, & \text{если $k$ чётно} \\
\frac{(2n)!}{2^{n-2} \cdot n!}, & \text{если $k$ нечётно}
\end{cases}
\]
\[
OA(n,k)=3^N, \; N=\binom{[k/2]-1}{n-3} + \binom{[k/2]-1}{n-4}+c, \; c=\begin{cases}
1, & \text{если $n=3,4$} \\
0, & \text{иначе}
\end{cases}
\]
\[
OS(n,k)=2^M, \; M=\sum^{n-2}_{m=1} \binom{[\md] + m - 1}{[\md]-1} - \sum^{n-2}_{m=1} \binom{[k/2] - 1}{m}
\]
\[
Cl(n,k)=\prod_{m=1}^{n-2} (m! \cdot 2^m + 1)^{L(m)}, \; L(m)=\binom{[k/2]-1}{m}\cdot \binom{n}{m}\cdot 2^{n-m}
\]
\end{theorem}

\begin{proof}
Количество попарно несвязанных состояний равно количеству различных значений полной системы инвариантов. Обозначим его за $S(n,k)$, тогда оно получается перемножением
\[
S(n,k)=BC(n,k) \cdot OA(n,k) \cdot OS(n,k) \cdot Cl(n,k),
\]
где
\begin{itemize}
    \item $BC(n,k)$ -- число различных значений тройки $(B_1,B_2,C)$,
    \item $OA(n,k)$ -- число различных значений $(O_1,\dots,O_N)$, $N$ -- количество знакопеременных класcов,
    \item $OS(n,k)$ -- число различных значений $(O_1,\dots,O_M)$, $M$ -- количество симметричных класcов,
    \item $Cl(n,k)$ -- число различных значений $(Cl_1,\dots,Cl_L)$, $L$ -- количество нетривиальных знакопеременных кластеров.
\end{itemize}

\textbf{Центральный инвариант и каркасные инварианты.}

Если $k$--чётно, то инвариантов нет. Если $k$ нечётно, то количество различных значений троек $(B_1,B_2,C)$
\[
2 \cdot 2 \cdot \frac{(2n)!}{2^n \cdot n!}=\frac{(2n)!}{2^{n-2} \cdot n!}.
\]

\textbf{Ориентационный инвариант.}

\textbf{1.} Пусть $\oG=A_m$.

Если $m \ge 5$, то коммутант совпадает с самой группой. При $m = 1,2$ группа вращения тождественна, но если
\begin{itemize}
    \item $m=3$, то $\oG/Q=A_3/[A_3,A_3]=A_3 \cong \Z_3$
    \item $m=4$, то $\oG/Q=A_4/[A_4,A_4]=A_4/H \cong \Z_3$, где $H$ -- четверная группа Клейна.
\end{itemize}

Значит, если этот инвариант определён для некоторого знакопеременного класса, то он принимает ровно три значения. Тогда
\[
OA(n,k)=3^N.
\]
Найдём количество знакопеременных классов $\class$ при $m=3,4$
\[
m=n \text{ или } 0<j_1<...<j_{n-m}<\md.
\]
Пусть соблюдено второе. Нужное количество равно количеству способов, которыми можно выбрать натуральные числа $j_i$, удовлетворяющие условию. Иначе говоря, число сочетаний из $([k/2]-1)$ (количество целых чисел в интервале $(0,\md)$) по количеству характеристик $(n-m)$
\[
\binom{[k/2]-1}{n-m}.
\]
И количество для $m=3$ и $m=4$ вместе взятых будет
\[
N=\binom{[k/2]-1}{n-3} + \binom{[k/2]-1}{n-4}=\binom{[k/2]}{n-3}
\]
по свойству биномиального коэффициента. Однако сокращенную запись мы использовать не будем, так как это равенство не всегда работает в наших случаях. Действительно, если $k=2$ или 3, то по договорённости оба левых биномиальных коэффициента обнуляются при любом $n$, в отличие от коэффициента справа. Упрощённая форма верна только для $k \geq 4$.

Пусть соблюдено первое условие. Тогда если $n=3$ или $n=4$, то необходимо посчитать дополнительный класс $m=n$. Итого
\[
N=\binom{[k/2]-1}{n-3} + \binom{[k/2]-1}{n-4}+c, \; c=\begin{cases}
1, & \text{если $n=3,4$} \\
0, & \text{иначе}
\end{cases}
\]

\textbf{2.} Пусть $\oG=S_m$.

Тогда для любого $m>1$
\[
S_m/[S_m,S_m]=S_m/A_m \cong \Z_2
\]
и
\[
OS(n,k)=2^M.
\]
Найдём M, количество всех симметрических классов при $1<m<n$.

Для начала, найдём общее количество классов в $(m), m<n$
\[
1 \leq j_1 \leq ... \leq j_{n-m} \leq \md.
\]
Это число сочетаний (с повторениями) из количества целых чисел на отрезке $[1,[\md]]$ по количеству характеристик $(n-m)$
\[
\binom{[\md] + n - m - 1}{n-m}=\binom{[\md] + n - m - 1}{[\md]-1}.
\]
Следовательно, симметрических классов в тривиальном классе
\[
\binom{[\md] + n - m - 1}{[\md]-1} - \binom{[k/2] - 1}{n-m}.
\]
Осталось лишь просуммировать эти величины по всем $1<m < n$
\[
M=\sum^{n-1}_{m=2} \binom{[\md] + n-m - 1}{[\md]-1} - \sum^{n-1}_{m=2} \binom{[k/2] - 1}{n-m} = \sum^{n-2}_{m=1} \binom{[\md] + m - 1}{[\md]-1} - \sum^{n-2}_{m=1} \binom{[k/2] - 1}{m} =
\]
\[
= \binom{[\md] + n - 2}{[\md]} - \sum^{n-2}_{m=0} \binom{[k/2] - 1}{m}
\]
по свойству биномиального коэффициента. В частности, если $n-2 \geq [k/2] - 1$ (то есть, при $k<2n$), то последняя сумма равна $2^{[k/2] - 1}$.

Однако в виду договорённостей, как и в предыдущем пункте, мы не можем использовать упрощённый результат в общем случае. Он годится лишь для $k \geq 4$.

\textbf{Кластерный инвариант.} Вспомним, что в $(m)$ находится
\[
\binom{[k/2]-1}{n-m}
\]
знакопеременных классов. Каждый из них разбивается на
\[
\binom{n}{m} \cdot 2^m
\]
кластеров, каждый со своим инвариантом. Всего знакопеременных кластеров в $(m)$
\[
L(m)=\binom{[k/2]-1}{n-m}\cdot \binom{n}{m}\cdot 2^m
\]
Помним, что угловые не делятся на кластеры и, значит, $m<n$. Тогда
\[
Cl(n,k)=\prod_{m=2}^{n-1} ((n-m)! \cdot 2^{n-m} + 1)^{L(m)}=
\prod_{m=1}^{n-2} (m! \cdot 2^m + 1)^{L(n-m)}.
\]
Переопределим $L(m)$ для большего удобства и получим утверждение теоремы.

\end{proof}

\subsection{Доказательство полноты}
Система полна, если она содержит максимальное число инварантов. Для доказательства полноты, необходимо показать, что если у двух произвольных состояний Кубика Рубика
\begin{itemize}
    \item разные значения полной системы, то они не связаны. Этот факт очевиден.
    \item одинаковые значения полной системы, то они связаны. Ниже приведём алгоритм, который в этом случае переводит головоломку из одного состояния в другое.
\end{itemize}

\textbf{1 шаг (каркасные инварианты).}

\textbf{Пусть $k$ нечётно. Переведём каркас из первого состояния во второе.} Далее будем называть ориентацией каркасной пары то, как располагаются её кубики, при неизменном положении пары. Легко видеть, что у каждой существуют лишь две ориентации.

Так как $B_1$ принимает одинаковые значения, то в обоих состояниях каркасные пары имеют один и тот же состав. Каждый поворот осуществляет транспозицию пар, поэтому мы можем реализовать любую их перестановку. Переставим их так, чтобы множества пар в двух состояниях отличались лишь ориентациями.

Для ориентирования пар заметим, что двойные повороты одновременно меняют ориентации двух пар. Тогда воспользуемся приёмом в замечании к теореме \ref{twocubesorient} и ориентируем все пары, кроме одной. Итак, все каркасные кубики (кроме, может быть, кубиков оставшейся пары) расположены на нужных местах.

Предположим, что последняя пара ориентирована неверно, то есть два каркасных кубика не на своих местах. Тогда чётности перестановки $\sigma_c$ в состояниях различны. В то же время перестановки $\sigma_p$ у них совпадают, как и значения $B_2$. Тогда должны совпадать и чётности $\sigma_c$, откуда следует противоречие. Значит, все пары ориентированы верно.

\textbf{2 шаг (центральный инвариант).}

\textbf{Переведём все центральные классы из положений в первом состоянии в положения во втором состоянии.}

\begin{deff}
Процесс, при котором мы переводим класс из начального положения в промежуточное, отличающееся от конечного чётной перестановкой элементов, называется \textbf{настройкой чётности} класса.
\end{deff}

Пусть $k$ нечётно. Лемма \ref{central} позволяет настроить чётность центрального класса $m$ вращением $(n+1-m)$-мерной внешней грани. При этом все центральные классы $m^{'}$, где $m^{'}<m$, остаются нетронутыми.

Центральный класс 1, то есть каркас, уже находится в нужном состоянии. Настроим чётность, если необходимо, центрального класса 2. Затем класса 3 и так далее, пока не доберёмся до $(n-1)$-го. Останется лишь класс угловых кубиков.

Центральные инварианты, как и чётности всех классов, кроме углового, совпадают. Тогда должны совпасть и чётности угловых. Таким образом, настройка завершена.

Если $k$ чётно, то центральный класс всего один -- угловой. Если необходимо настроить его чётность, то совершаем один поворот двумерного внешнего слоя. Четыре угловых кубика циклически поменяются и чётность будет настроена.

Применяя теорему о чётной перестановке класса, завершаем шаг 2. 

\textbf{3 шаг.}

Все классы далее не являются центральными.

\textbf{Выберем произвольный класс и переведём его в положение второго состояния.} Для этого занумеруем его кубики в первом состоянии и отметим, где они должны находиться во втором. Если их перестановка чётная, то существует нужная аккуратная комбинация по теореме о чётной перестановке.

Пусть оказалось, что перестановка нечётная. Вспомним, что в нецентральных классах все кубики не уникальны. Выберем два неотличимых и поменяем их номера в первом состоянии местами (см. рис. \ref{sec6sub2pic1}). Тогда необходимая перестановка класса снова будет чётной, и можно будет воспользоваться теоремой.

\begin{figure}[h]
    \centering
    \includegraphics[width=0.7\linewidth]{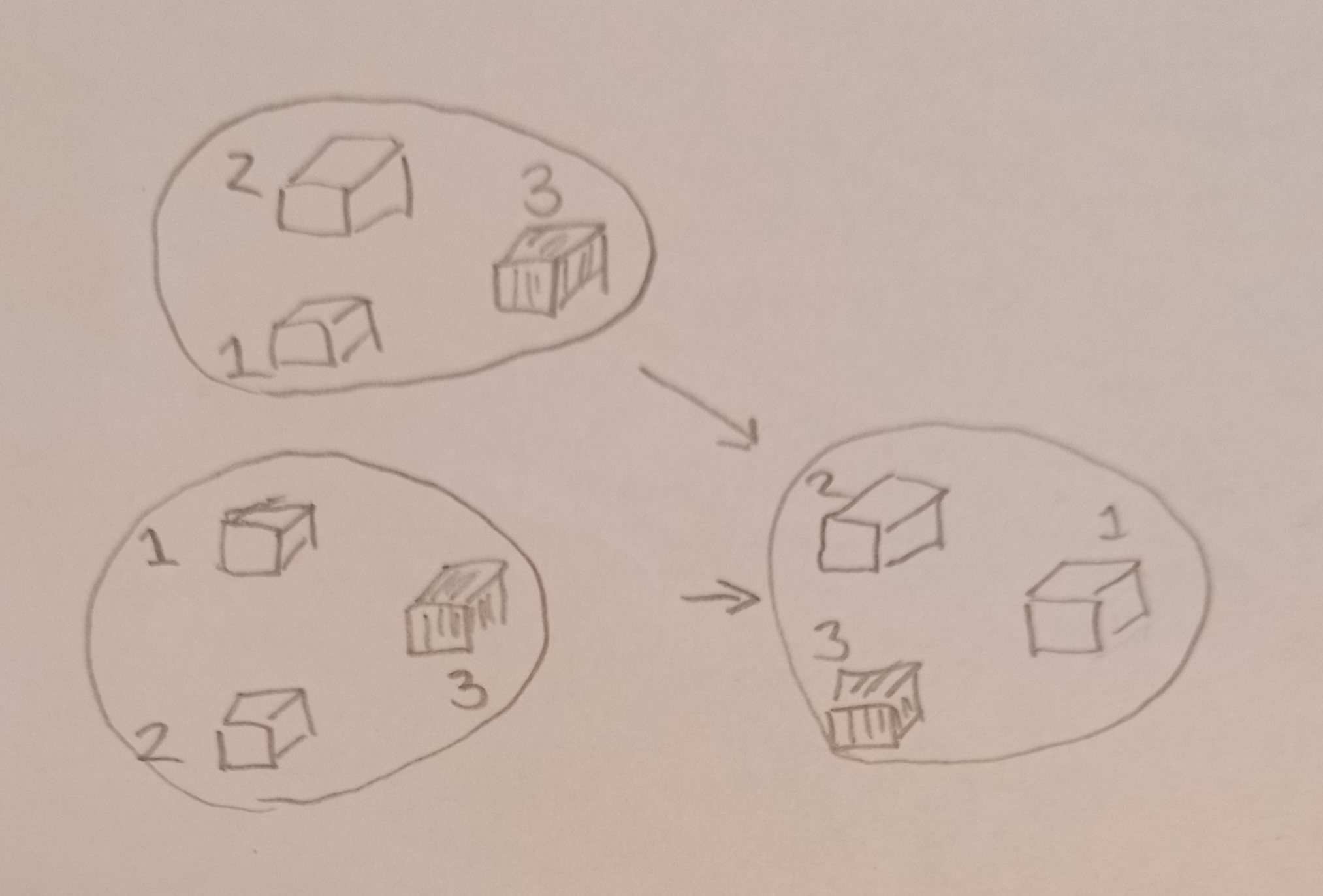}
    \caption{}
    \label{sec6sub2pic1}
\end{figure}

\textbf{4 шаг (кластерные инварианты).}

\textbf{Выберем произвольный нетривиальный знакопеременный класс.} Его кластеры уже находятся в нужных положениях, однако кубики ещё не ориентированы. Совершим дополнительный шаг.

Перейдём к упрощённым ориентациям. Переставим элементы класса так, чтобы каждый кластер остался в своем положении, но при этом был упрощённо ориентирован так же, как и во втором состоянии.

Выберем кластер $Cl$ из класса. В двух его состояниях совпадают ориентационные инварианты. Значит, в них одинаковое количество кубиков с $or=1$. То же верно и для $or=0$. Этот факт допускает возможность ориентирования кластера (см. рис. \ref{sec6sub2pic2}). Пусть $sgn(\sigma_{Cl})$ -- чётность перестановки, которую нужно совершить, чтобы ориентировать кластер.

\begin{figure}[h]
    \centering
    \includegraphics[width=0.5\linewidth]{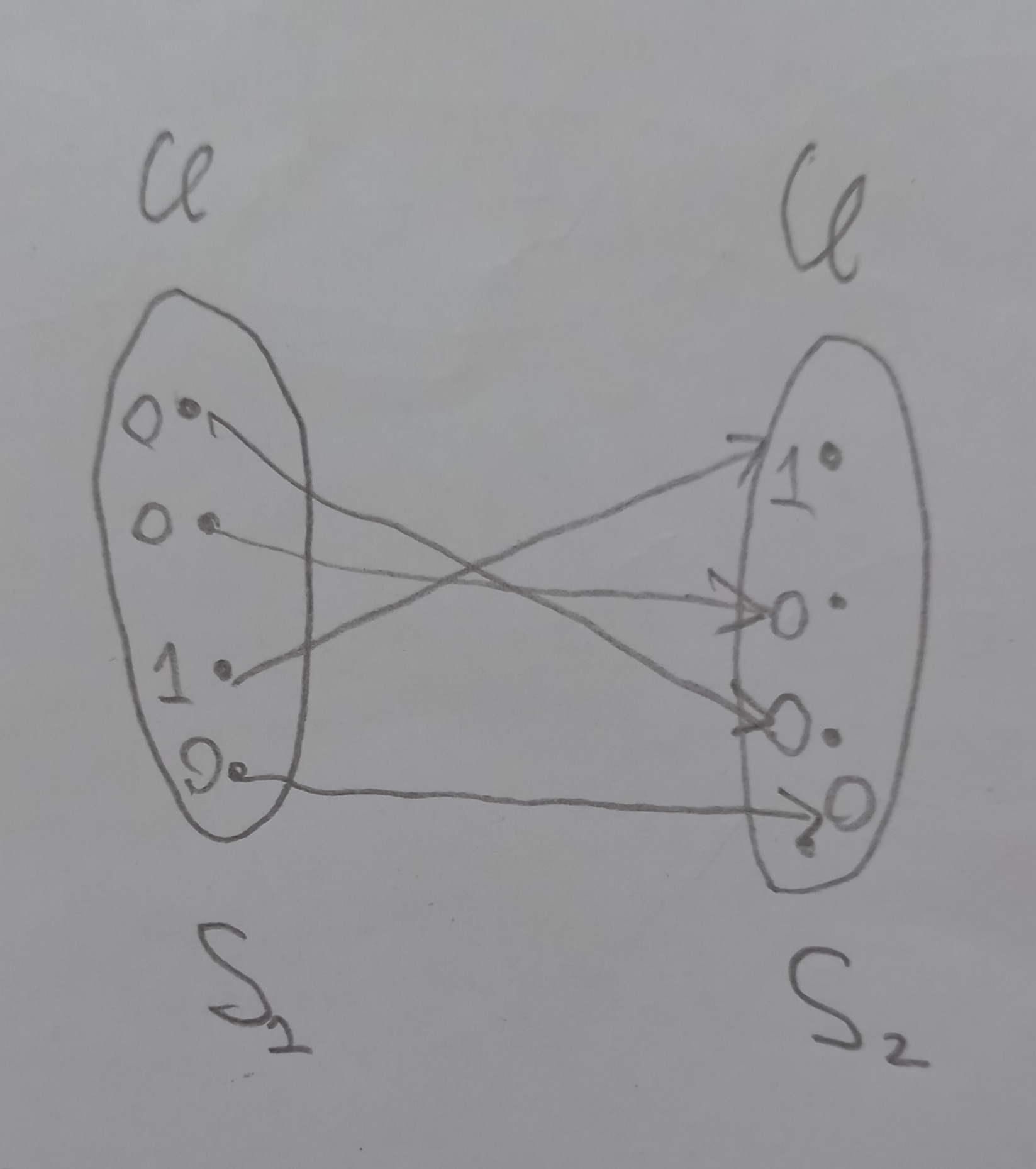}
    \caption{}
    \label{sec6sub2pic2}
\end{figure}

Если 
\[
\prod_{Cl} sgn(\sigma_{Cl})=1,
\]
то чтобы упрощённо ориентировать класс нужна чётная перестановка. Далее пользуемся соответствующей теоремой и аккуратно переставляем все элементы.

Пусть произведение равно -1. Если в классе найдутся два кубика из одного кластера с одинаковой упрощённой ориентацией, то они неотличимы и по раскраске и по ориентации. Воспользуемся приёмом из третьего шага и сменим чётность перестановки.

В каких случаях такие кубики могут не существовать? В каждом нерёберном кластере есть хотя бы 8 элементов. Среди них наверняка есть два, совпадающих по упрощённой ориентации. У рёберных же кластеров всего по два кубика.

Тогда исключительным случаем будет рёберный класс, в котором у каждого кластера ориентации кубиков 1 и 0. В такой ситуации придётся настраивать чётность рёберного класса другим способом, чем в шаге 3.

Выберем произвольный внешний трёхмерный слой Кубика Рубика. Для удобства спроецируем его на трёхмерное пространство (см. рис. \ref{sec6sub2pic3}). Если нужно настроить чётность нецентрального класса $(n-1,j), j \neq \mdd$, то повернём грань, как на рисунке. При таком повороте циклически переставятся 4 кубика класса. Что самое главное, ни один из центральных или других рёберных кубиков не будет задет. 

\begin{figure}[h]
    \centering
    \includegraphics[width=0.5\linewidth]{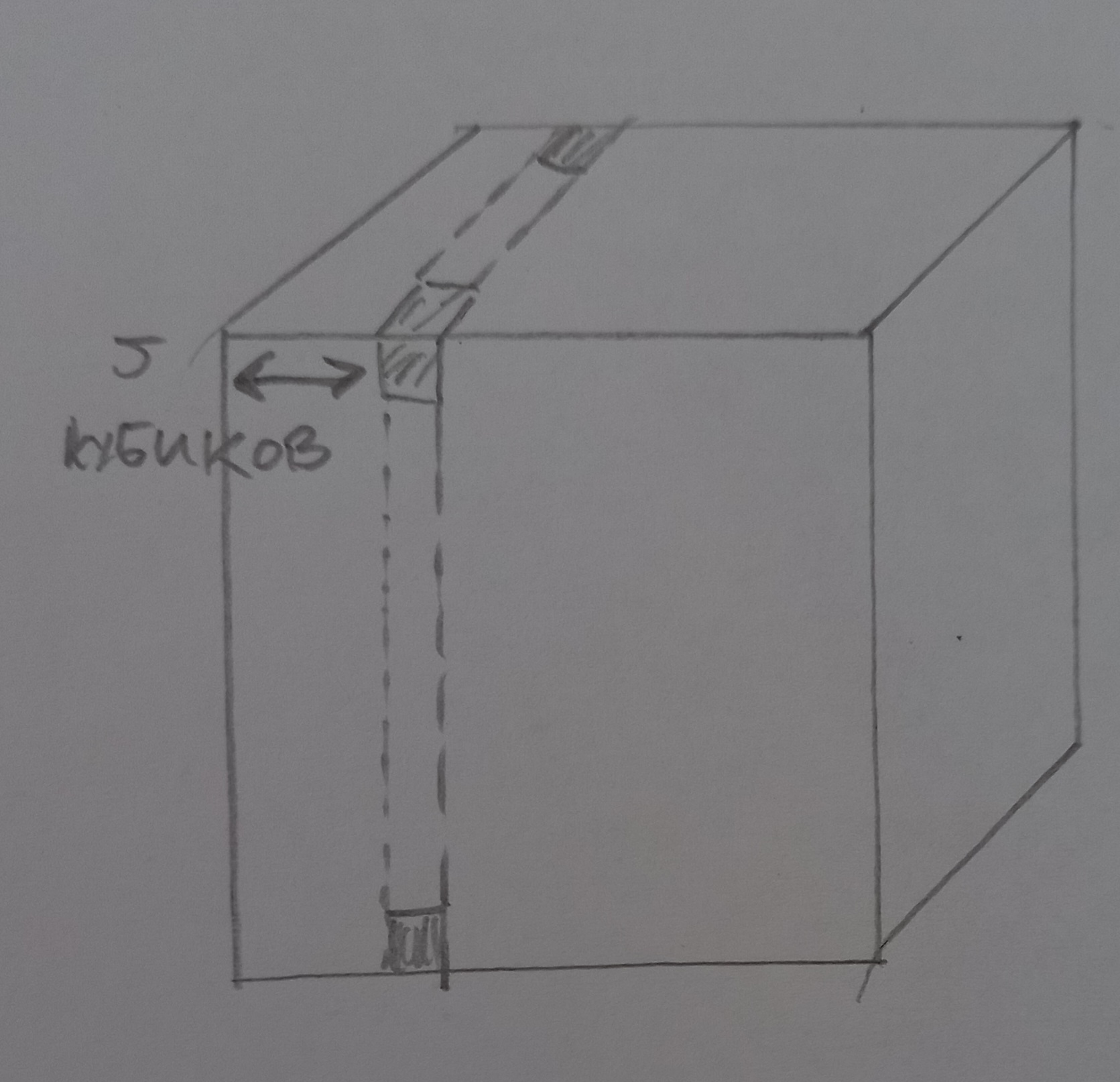}
    \caption{}
    \label{sec6sub2pic3}
\end{figure}

Проведём эту настройку чётностей рёберных перед третьим шагом.

Подведём итог
\begin{itemize}
    \item Сначала перевели каркас во второе состояние.
    \item Аккуратно расположили центральные классы также, как и во втором состоянии. Остались только нецентральные.
    \item По необходимости настроили чётности нецентрально-рёберных кубиков. При этом состояния центральных классов изменены не были.
    \item Аккуратно расположили все нецентральные классы как во втором состоянии.
    \item Аккуратно упрощённо соориентировали все нетривиальные знакопеременные классы. В их числе и нецентрально-рёберные, для которых был совершен дополнительный шаг.
\end{itemize}

Таким образом, осталось только ориентировать классы.

\textbf{5 шаг (ориентационные инварианты)}

У всех классов в обоих состояниях совпадают ориентационные инварианты. Воспользуемся теоремой об ориентировании класса и завершим доказательство полноты. Но для её применения необходимо проверить, что ориентация каждого кубика во втором состоянии достижима из ориентации первого состояния.

\begin{itemize}
    \item У всех симметрических классов зависимая группа вращения -- симметрическая, потому ориентация во втором состоянии каждого их элемента достижима из ориентации в первом. Помним, что все центральные классы кроме углового являются симметрическими.
    \item У угловых кубиков зависимая группа хоть и не является симметрической, но совпадает с независимой. Поэтому их ориентации в одном состоянии достижимы из другого.
    \item Остались знакопеременные классы. Кубики каждого такого класса в обоих состояниях имеют одну и ту же упрощённую ориентацию. Тогда их ориентации достижимы друг из друга.
\end{itemize}

\end{document}